\documentclass[12pt,a4paper]{amsart}
\usepackage{amsfonts}
\usepackage{amssymb}
\usepackage{amsmath}
\usepackage{amsthm}
\usepackage{cite}
\usepackage{graphicx}
\usepackage{amscd}
\usepackage{color} 
\newtheorem{theorem}{Theorem}
\theoremstyle{plain}

\newtheorem{definition}[theorem]{Definition}

\newtheorem{lemma}[theorem]{Lemma}

\newtheorem{proposition}[theorem]{Proposition}
\newtheorem{remark}[theorem]{Remark}
\numberwithin{equation}{section}

\def\eps{\varepsilon}

\newcommand{\R}{\mathbb{R}}

\newcommand{\N}{\mathbb{N}}

\renewcommand{\phi}{\varphi}

\newcommand{\dist}{\text{\rm dist}}
\newcommand{\supp}{\text{\rm supp}}

\newcommand{\esssup}{\text{\rm ess\,sup}}

\begin{document}

\allowdisplaybreaks

\title[Distributional solutions of the nonlinear Schr\"odinger equation]{Distributional solutions
of the stationary nonlinear Schr\"odinger equation: singularities, regularity and
exponential decay}

\author{Rainer Mandel}
\address{R. Mandel \hfill\break 
Institut f\"ur Analysis, Karlsruhe Institute of Technology (KIT)\hfill\break
D-76128 Karlsruhe, Germany}
\email{Rainer.Mandel@kit.edu}

\author{Wolfgang Reichel}
\address{W. Reichel \hfill\break 
Institut f\"ur Analysis, Karlsruhe Institute of Technology (KIT), \hfill\break
D-76128 Karlsruhe, Germany}
\email{wolfgang.reichel@kit.edu}

\date{\today}

\subjclass[2000]{Primary: 35Q55, 35J20; Secondary: 35J08, 35J10}
\keywords{nonlinear Schr\"odinger equation, singular solutions, variational methods, distributional
solutions}

\begin{abstract}  
  We consider the nonlinear Schr\"{o}dinger equation $-\Delta u + V(x) u = \Gamma(x)
  |u|^{p-1}u$ in $\R^n$ where the spectrum of $-\Delta+V(x)$ is positive. 
   In the case  $n\geq 3$ we use variational methods to prove
  that for all $p\in (\frac{n}{n-2},\frac{n}{n-2}+\eps)$ there exist
  distributional solutions with a point singularity at the origin provided $\eps>0$ is sufficiently
  small and $V,\Gamma$ are bounded on $\R^n\setminus B_1(0)$ and satisfy suitable H\"{o}lder-type
  conditions at the origin. In the case $n=1,2$ or $n\geq
  3,1<p<\frac{n}{n-2}$, however, we show that every distributional solution of the more general
  equation $-\Delta u + V(x) u = g(x,u)$ is a bounded strong solution if
  $V$ is bounded and $g$ satisfies certain growth conditions.
\end{abstract}

\maketitle


\section{Introduction and main result}

In this paper we investigate distributional solutions of the stationary nonlinear
Schr\"{o}dinger equation (NLS)
\begin{align} \label{Gl I}
  -\Delta u + V(x) u = \Gamma(x) |u|^{p-1}u\quad\mbox{in }\R^n
\end{align}
for $n\in\N$ and $1< p < \frac{n+2}{(n-2)_+}$. The NLS \eqref{Gl I} has
been receiving much attention due to its applicability in different fields of mathematical physics,
e.g. nonlinear optics, mean field theory, Bose-Einstein condensates. Spatially localized
soliton-like solutions $u\in H^1(\R^n)$ of \eqref{Gl I} can be expected whenever $0$ does not
belong to the spectrum of $-\Delta + V(x)$. Ever since pioneering work of Strauss~\cite{strauss},
Berestycki-Lions~\cite{berlio,berlio2}, Stuart~\cite{stuart} a lot of results on existence and
non-existence of ground states/bound states, multiplicity, asymptotic behaviour, bifurcation
phenomena etc. have been obtained. In the case where $V, \Gamma$ are positive constants the results
of Gidas, Ni, Nirenberg~\cite{GNN} and Li~\cite{li} apply and show that all positive
solutions decaying to $0$ at infinity must be radially symmetric. Recently, due to new
developements in photonic crystals, the case of periodic coefficients $V,\Gamma$ has been studied,
cf. Pankov~\cite{Pankov} and Szulkin-Weth~\cite{Szulkin_Weth}. In all of these works the solutions were weak (or classical)
solutions belonging to $H^1(\R^n)$.

\smallskip

More recently, distributional solutions of nonlinear elliptic boundary value problems like
\eqref{Gl I} have been studied. In the context of bounded domains various classes of {\em very weak
solutions}, i.e. subclasses of distributional solutions with prescribed Dirichlet boundary data,
have been investigated, cf. Stampacchia~\cite{Stam}, Br\'{e}zis et al.~\cite{BCMR},
Quittner-Souplet~\cite{QS}, McKenna-Reichel~\cite{mcr}, McKenna et al.~\cite{horreimck}, del Pino
et al. \cite{dPMP}. In the context of the Yamabe problem, Pacard~\cite{Pac, Pac2} and
Mazzeo-Pacard~\cite{mazpac} have also studied distributional solutions of nonlinear boundary value
problems similar to \eqref{Gl I}. In many of the above mentioned results the following phenomenon
occurs: for a range of exponents $1<p<p^\ast$ all very weak solutions turn out to have no
singularities and are indeed bounded weak/classical
solutions of the nonlinear elliptic problem, whereas for
$p^\ast<p<p^\ast+\eps$ unbounded very weak solutions were shown to exist.

\smallskip

In the present paper we show a similar phenomenon for the NLS \eqref{Gl I}. The singular
distributional solutions that we find have some properties in common with $H^1(\R^n)$-solutions of
\eqref{Gl I}, e.g. they decay exponentially fast at infinity. On the other hand, even in cases
where there are no non-trivial $H^1(\R^n)$ solutions, singular distributional solutions can be
shown to exist, cf. Remark \ref{rem I}. Let us point out two further interesting aspects of
singular distributional solutions of \eqref{Gl I}: First, if $V, \Gamma$ satisfy the conditions
given below and are radially symmetric such that $\Gamma$ is positive and radially decreasing and
$V$ is positive and radially increasing then by Li's result, cf.
\cite{li}, all weak/classical non-negative solutions which decay to $0$ at infinity must be
radially symmetric. 
However, using Theorem~\ref{Thm unbounded solution} one can construct a distributional
solution which is not radially symmetric having a single point singularity at the origin although
$V,\Gamma$ are radially symmetric with respect to some point $x_0\in\R^n\setminus\{0\}$. 
Second,
let us view singular distributional solutions from the point of view of numerical approximations. 
From the outcome of one numerical calculation of an approximate solution to \eqref{Gl I} it is impossible to tell if the computed
result approximates a singular disitributional solution or a very large weak/classical solution.
Mesh refinements may help to clarify it. However, from our Theorem~\ref{Thm regularity} it is clear
that below the exponent $p^\ast=\frac{n}{n-2}$ (which is smaller than the usual critical exponent
$\frac{n+2}{n-2}$) no such singular distributional solutions can exist.

\smallskip

Our tools range from linear Schr\"odinger theory, calculus of variations, Green's functions to the
use of singular integral estimates. Results concerning exponential decay of eigenfunctions are
proved by an adapted version of Agmon's method (cf.
\cite{hislopsigal},\cite{hundertmark},\cite{hunzikersigal}).

\medskip

In our first result Theorem~\ref{Thm unbounded solution} we follow the ideas of
\cite{horreimck}, \cite{mazpac} to prove the existence of an
unbounded exponentially decaying distributional solution of \eqref{Gl I} when $n\geq 3$ and 
$\frac{n}{n-2}<p<\frac{n}{n-2}+\eps$ for $\eps>0$ sufficiently small.
We concentrate on the construction of distributional solutions with one point singularity at the
origin. To this end we assume the following conditions on $V,\Gamma:\R^n\to\R$:
\begin{enumerate}
  \item[(H1)] $V\in L^\infty\big(\R^n\setminus B_1(0)\big)$ and
  there are constants $C_1>0$ and $\alpha \geq \frac{n-6}{2}$ such that 
  $$
    |V(x)| \leq C_1 |x|^\alpha \quad\mbox{ for almost all }  x\in B_1(0).
  $$
  \item[(H2)] $\Sigma:= \min\sigma(-\Delta +V(x))>0$ where $\sigma$ denotes the
  $L^2$-spectrum.
  \item[(H3)] $\Gamma\in L^\infty(\R^n)$ and there are constants $C_2>0$ and
  $\beta>\frac{n-2}{2}$ such that 
  $$
      |\Gamma(x)-\Gamma(0)|
      \leq C_2 |x|^\beta \quad\mbox{ for almost all } x\in B_1(0),
  $$
  where $\Gamma(0)>0$.  Rescaling \eqref{Gl I} we can assume w.l.o.g.
  $\Gamma(0)=1$.
\end{enumerate}

\medskip 

In our second result Theorem~\ref{Thm regularity} we show that for $1<p<\frac{n}{(n-2)_+}$
the equation
\begin{align}\label{Gl II}
  -\Delta u + V(x) u = g(x,u) \quad\mbox{in }\R^n
\end{align}
and in particular \eqref{Gl I} does not admit positive locally
unbounded distributional solutions provided $g:\R^n\times\R\to \R$ is a Carath\'{e}odory function
which satisfies
\begin{align} \label{Gl growth cond I} 
  -C_3+C_4 s^p \leq g(x,s)\leq C_3+C_5s^p\qquad (x\in\R^n,s\geq 0).      
\end{align}
where $C_3,C_4,C_5>0$. We also obtain a global boundedness and a global regularity result in the
case $g$ satisfies
\begin{align}
  |g(x,s)| \leq C_6\,(|s|+|s|^p) \qquad (x\in\R^n, s\in\R), \label{Gl growth cond II}      
\end{align}
where $C_6>0$. In addition we find that distributional solutions of \eqref{Gl II} decay
exponentially in the case
\begin{align}
  \lim_{s\to 0} \underset{x\in\R^n}{\esssup} \,\frac{|g(x,s)|}{|s|} = 0. \label{Gl growth cond III}
\end{align}
If remains open if or if not unbounded distributional solutions exists in the borderline case $p=\frac{n}{n-2}$.

\medskip

All our results are built on the following notion of a distributional solution.
\begin{definition} \label{Def very weak solution}  
  Let $g:\R^n\times \R\to \R$ be a
  Carath\'{e}odory function with $|g(x,s)|\leq C(1+|s|^p)$ for all $s\in \R$, almost all $x\in\R^n$ and some $C>0$, $1<p<\infty$. A function $u\in L^p_{loc}(\R^n)$ with $Vu\in L^1_{loc}(\R^n)$ is called a distributional solution of \eqref{Gl II} if $$
    \int_{\R^n} u(-\Delta\phi+V(x) \phi)\,dx = \int_{\R^n} g(x,u)\phi \,dx \quad \mbox{ for all }
    \phi\in C_c^\infty(\R^n). 
  $$
\end{definition}
In contrast, a function  $u\in
L^p_{loc}(\R^n)$ with $\nabla u, Vu \in L^1_{loc}(\R^n)$ is called a {\em weak solution} of \eqref{Gl II} if
\begin{equation}
 \int_{\R^n} \nabla u \nabla \phi+V(x) u\phi\,dx = \int_{\Omega} g(x,u)\phi \,dx \quad \mbox{ for all }
    \phi\in C_c^\infty(\R^n).
\label{def_weak}
\end{equation}
Similarly, we say that $u$ is a {\em weak solution of \eqref{Gl II} on an open
subset $\Omega\subset\R^n$} if \eqref{def_weak} holds for all $\phi\in C_c^\infty(\Omega)$.
A function $u\in L^p_{loc}(\R^n)$ with $-\Delta u, Vu \in L^1_{loc}(\R^n)$ will be called a
{\em strong solution} of \eqref{Gl II} if $-\Delta u + Vu=g(x,u)$ holds almost everywhere in $\R^n$.


\medskip

Our main results are the following two theorems.

\begin{theorem}[Supercritical case] \label{Thm unbounded solution}
  Let the assumptions (H1),(H2),(H3) hold and let $n\geq 3$. Then there
  exists $\eps>0$ such that for all $p \in (\frac{n}{n-2},\frac{n}{n-2}+\eps)$ there is a
  distributional solution $U$ of \eqref{Gl I} with the following properties:
  \begin{itemize}
    \item[(i)] $\esssup_{B_\delta(0)} U=+\infty$ for all $\delta>0$ and $U\in
    L^q(\R^n)$ for all $1\leq q<\frac{n(p-1)}{2}$. 
    \item[(ii)] For all $\delta>0$ the function $U\in H^1(\R^n\setminus B_\delta)$ is a weak
    solution of \eqref{Gl I} on $\R^n\setminus B_\delta$.
    \item[(iii)] For all $\mu\in (0,\sqrt{\Sigma})$ there is $C_\mu>0$ such that $|U(x)|\leq C_\mu
    e^{-\mu|x|}$ if  $|x|\geq 1$.
    \item[(iv)] If in addition $\Gamma\geq 0$ then $U$ can be chosen to satisfy $U\geq 0$. 
  \end{itemize}  
\end{theorem}

\begin{theorem}[Subcritical case] \label{Thm regularity}   
  Let $n\in\N$, $1< p <\frac{n}{(n-2)_+}$, $V\in
  L^\infty(\R^n)$, let $g:\R^n\times \R\to \R$ be a Carath\'{e}odory function and let $u$ be a
  distributional solution of \eqref{Gl II}.
  \begin{enumerate}    
    \item {\em (Local regularity)} If $g$ satisfies \eqref{Gl growth cond I} and if $u\geq 0$ then
    $u\in W^{1,\infty}_{loc}(\R^n)\cap W^{2,q}_{loc}(\R^n)$ for all $q\in [1,\infty)$.
    \item {\em (Global regularity)} If $g$ satisfies \eqref{Gl growth cond
    II} and if $u\in L^p(\R^n)$ then $u\in
    W^{1,q}(\R^n)\cap W^{2,q'}(\R^n)$ for all $q\in [p,\infty],q'\in [p,\infty)$.   
    If in addition $V$ satisfies (H2) and $g$
    satisfies \eqref{Gl growth cond III} then $u\in W^{1,q}(\R^n)\cap W^{2,q'}(\R^n)$ for all $q\in
    [1,\infty],q'\in (1,\infty)$ and for all $0<\mu<\sqrt{\Sigma}$ there is $C_\mu>0$ such that $|u(x)|\leq C_\mu
    e^{-\mu|x|}$ in $\R^n$.
  \end{enumerate}  
  In both cases $u$ is a strong solution of \eqref{Gl II}.
\end{theorem}

\begin{remark} \label{rem I}~
  \begin{enumerate}
    \item Note that for every compact set $K\subset\R^n$ the potential $V =
    1_{\R^n\setminus K}$ satisfies (H1),(H2) for every $\alpha\geq \frac{n-6}{2}$.
	\item In the case $n=3,4,5,6$ Theorem \ref{Thm unbounded solution} applies to
	every measurable function $V$ which satisfies $0<V_0\leq V\leq V_1$ almost everywhere for some
	positive constants $V_0,V_1$. For instance we find an unbounded distributional solution
	of the equation $-\Delta u + V(x)u = |u|^{p-1}u$ in the case $V\in W^{1,\infty}(\R^n)$ is
	strictly monotone in some direction $v\in\R^n$, e.g. $V(x)=\pi+\arctan(xv)$. This is quite
	interesting given the fact that in this case the only $H^1(\R^n)$-solution is the trivial one.
	Indeed, if $u\in H^1(\R^n)$ is a solution then $u\in H^2(\R^n)$ (see Theorem \ref{Thm
	regularity},(2)) and testing the equation with $\partial_v u$ leads to
	\begin{align*}
	  0 
	  &= \int_{\R^n} \Big( \nabla u \nabla (\partial_v u) + V u\partial_v u - |u|^{p-1}u\partial_v
	  u\Big)\,dx \\
	  &= \int_{\R^n} \partial_v \Big( \frac{1}{2}|\nabla u|^2 - \frac{1}{p+1} |u|^{p+1} \Big)
	    + \frac{1}{2} \int_{\R^n} V\partial_v (|u|^2)\,dx \\
	  &= -\frac{1}{2} \int_{\R^n} (\partial_v V) |u|^2\,dx
	\end{align*}
	by density of $C_0^\infty(\R^n)$ in	$H^2(\R^n)$. Hence, $u\equiv 0$ because $\partial_v V<0$ in $\R^n$.
	The above result is due to Tanaka~\cite{tanaka}, see also Theorem 1.3. in \cite{ikoma}.	
    \item If we add regularity assumptions on $V$ and $g$ in Theorem \ref{Thm regularity} then
    elliptic regularity theory will give better results. If $V$ and $g$ are both
    $C^\infty$-functions, say, then every positive distributional solution $u$ of \eqref{Gl II} is
    in fact a classical solution. Similarly, if in Theorem \ref{Thm unbounded solution} $V,g$ are
    both $C^\infty$-functions then part (ii) of Theorem  \ref{Thm unbounded solution} gives $U\in
    C^\infty(\R^n\setminus\{0\})$. 
    \item  By a suitable choice of test functions one can extend the local regularity
    result of Theorem~\ref{Thm regularity} to possibly sign-changing solutions of equation \eqref{Gl
    II} where the nonlinearity satisfies the more general inequality $|g(x,u)|\leq
    c(1+|u|^p)$.
  \end{enumerate}  
\end{remark}~

In the proof of Theorem \ref{Thm unbounded solution} we always require $0<\eps < \frac{2}{n-2}$
so that $\frac{n}{n-2}<p<\frac{n+2}{n-2}$ and variational methods are applicable.
Estimates involving $p-\frac{n}{n-2}$ will be carried out explicitly. Throughout the paper $B_r =
\{x\in\R^n : |x|<r\}$ is the open ball of radius $r$ in $\R^n$ and $c$ is a constant which can
change from line to line but which is independent of $p$. We use the symbol $\frac{n}{(n-2)_+}$ to
denote the value $\infty$ for $n=1,2$ and the value $\frac{n}{n-2}$ in the case $n\geq 3$.
Similarly the symbols $\frac{n}{(n-1)_+}$, $\frac{2n}{(6-n)_+}$ etc. are used. The assumptions
(H1), (H2) imply that the bilinear form
    \begin{align}\label{Gl Def innerproduct}
      \langle u,v \rangle_V := \int_{\R^n} \big( \nabla u\nabla v + V(x)uv \big)\,dx 
      \qquad (u,v \in H^1(\R^n))
    \end{align}
generates a norm $\|\cdot\|_V$ on $H^1(\R^n)$ which is equivalent to the standard $H^1$-norm $\|\cdot\|$. 

\medskip

 Finally let us recall
the definition of the Kato class $K_n$, cf. \cite{sim}. Let $h_n(x,y)=|x-y|^{2-n}$ for
$n\geq 3$, $h_2(x,y)=-\log|x-y|$ and $h_1(x,y)=1$. A measurable function
$W:\R^n\to \R$ belongs to $K_n$, $n\in \N$ if
\begin{align*}
\lim_{\rho\to 0} \sup_{x\in \R^n} \int_{\{|x-y|\leq \rho\}} h_n(x,y)|W(y)|\,dy&=0\;\,, \quad n\geq 2,\\ 
\sup_{x\in \R^n} \int_{\{|x-y|\leq 1\}}
|W(y)|\,dy&<\infty, \quad n=1.
\end{align*}
A norm on $K_n$ is given by (cf. \cite{sim}, p.453, (A15))
$$
\|W\|_{K_n} := \sup_{x\in \R^n} \int_{\{|x-y|\leq 1\}} h_n(x-y) |W(y)|\,dy.
$$
If $\Omega\subset\R^n$ is open we denote by $K_n(\Omega)$ the set of measurable functions $W:\R^n\to\R$ such that $W1_\Omega$
  lies in the Kato class $K_n$. The mapping $\|W\|_{K_n(\Omega)}:= \|W1_\Omega\|_{K_n}$ defines a
  seminorm on $K_n(\Omega)$. For every $q\in (\frac{n}{2},\infty]$ there exists a constant $c_q>0$ such
  that
  \begin{align}\label{Gl KatoNorm}
    \|W\|_{K_n(\Omega)}    
    \leq c_q \sup_{y\in\Omega} \|W\|_{L^q(B_1(y))}
  \end{align}   
  whenever the right hand side is finite. 


\section{Proof of Theorem \ref{Thm unbounded solution}}

  Our existence proof of an unbounded distributional solution $U$ is inspired
  by \cite{horreimck},\cite{mazpac}.   
  We start by constructing an approximate
 solution $u_0$ of equation \eqref{Gl I} which is unbounded near $0$. Then we determine a functional
  $J:H^1(\R^n)\to\R$ such that every critical point $\tilde{u}\in H^1(\R^n)$ of $J$ gives rise to a
  distributional solution $U:= u_0+\tilde{u}$ of \eqref{Gl I} which has the desired properties. 
  The main difficulty will be to prove that $J$ has a critical point.
  The proof of the parts (i) and (ii), (iii), (iv) will be given in section \ref{distrib_sec},
  \ref{decay_sec}, \ref{pos_sec} respectively.

  \subsection{Construction of an unbounded approximate solution}

  For $p > \frac{n}{n-2}$ let the function $u_1\in C^\infty(\R^n\setminus\{0\})$ be defined by
    \begin{align}\label{Gl Def u1}
      u_1(x) := c_{n,p} |x|^{-\frac{2}{p-1}}\quad\mbox{where}\quad
      c_{n,p} = \big(\frac{2}{p-1}(n-2-\frac{2}{p-1})\big)^{\frac{1}{p-1}}.
    \end{align}
    Notice that $c_{n,p}\to 0$ as $p\searrow \frac{n}{n-2}$ and
    \begin{align} \label{Gl PDE u1}
      -\Delta u_1=u_1^p \qquad\text{in } \R^n\setminus\{0\}.
    \end{align}
    Replacing $u_1$ outside a suitable ball $B_\rho$ by an exponentially decreasing classical solution $u_2$
    of
     \begin{align} \label{Gl PDE u2}
      -\Delta u_2 + u_2 =  u_2^p \qquad\text{in } \R^n\setminus B_\rho
    \end{align}
    we define the approximate solution
	\begin{align}\label{Gl App sol}
         u_0(x) :=
      \begin{cases}
        u_1(x) , &x\in B_\rho, \\
        u_2(x) , &x\in \R^n\setminus B_\rho.
      \end{cases}
    \end{align}
    It turns out that such a function $u_0$ can be constructed with properties stated next. To
    state the Proposition let us define 
    $$
     \partial_\nu^+ u_0(x) = \lim_{t\to 0+} \frac{u_0(x)-u_0(x-t\nu(x))}{t}, \qquad \partial_\nu^-
     u_0(x) = \lim_{t\to 0+} \frac{u_0(x+t\nu(x))-u_0(x)}{t} 
     $$
    for $\nu(x)=\frac{x}{|x|}$ whenever the limits exist.

    \begin{proposition}[Existence of an approximate solution] \label{Prop Existence approximate
    solution}
      Let $n\in\N,n\geq 3$. Then there exists a radius $\rho\geq 1$ and a constant $c>0$ such that
      for all $p\in (\frac{n}{n-2},\frac{n+2}{n-2})$ there is a radially symmetric function
      $u_0:\R^n\setminus\{0\}\to (0,\infty)$ with the following properties:
      \begin{enumerate}
        \item[(i)] $u_0\in C^2(B_\rho\setminus\{0\})$ solves \eqref{Gl PDE u1} in
        $B_\rho\setminus\{0\}$ in the classical sense.
        \item[(ii)] $u_0\in C^2(\R^n\setminus \overline{B}_\rho)$ solves \eqref{Gl PDE u2} in
        $\R^n\setminus \overline{B}_\rho$ in the classical sense.             
        \item[(iii)] $u_0\in C(\R^n\setminus\{0\})$ and all first and second order
        derivatives of $u_0$ admit continuous extensions to $\partial B_\rho$ from either side.
        Moreover, for all $\delta>0$ we have $u_0 \in
        H^1(\R^n\setminus B_\delta)$.   
        \item[(iv)] $\lim_{x\to 0} u_0(x) = +\infty$.
        \item[(v)] $|\partial_\nu^+ u_0(x)-\partial_\nu^- u_0(x)| \leq c\, c_{n,p}$ for all $x\in
        \partial B_\rho$.      
        \item[(vi)] $u_0$ satisfies the estimate
        \begin{equation}
      	u_0(x) \leq \left\{\begin{array}{ll}
       	c_{n,p} |x|^{-\frac{2}{p-1}} & \mbox{for } x\in B_\rho, \vspace{\jot} \\
     	c_{n,p} e^{-\frac{|x|-\rho}{2}} & \mbox{for } x\in \R^n\setminus B_\rho.
     	\end{array}\right.  \label{Gl_estimate}
    	\end{equation}
    	In particular, $u_0\in L^q(\R^n)$ for all $q\in [1,\frac{n(p-1)}{2})$.
      \end{enumerate}     
    \end{proposition}
    
    For a  proof of this result we refer to Appendix A.

    \subsection{Variational setting} \label{var_set_sec}

    Given $u_0$ from Proposition~\ref{Prop Existence approximate solution} we prove existence of
    an unbounded distributional solution $U$ of \eqref{Gl I} using the ansatz
    \begin{align}\label{Gl Def uv}
	  U := u_0 + \tilde{u}
    \end{align}
    where $\tilde{u}\in H^1(\R^n)$ will be constructed as a local minimizer of a suitable functional
    $J:H^1(\R^n)\to\R$. Once the existence of $\tilde u$ is shown we will see that
    $U:=u_0+\tilde u$ is a weak solution of \eqref{Gl I} on $\R^n\setminus B_\delta$ for every
    $\delta>0$ and a distributional solution of \eqref{Gl I} on $\R^n$. The definition of $J$
    stems from the following motivation.

    \medskip
    
    For a fixed test function $\phi \in C_0^\infty(\R^n\setminus\{0\})$ we
    have by Proposition \ref{Prop Existence approximate solution} 
   \begin{align}      \label{Gl Equation u0}
    \int_{\R^n} (\nabla u_0 \nabla \phi + V(x) u_0 \phi)\,dx
     &= \int_{\R^n} u_0^p\phi \,dx  + \oint_{\partial B_\rho} (\partial_{\nu}^+ u_0 -
       \partial_{\nu}^- u_0)\phi \,d\sigma \\ &\quad+ 
     \int_{B_\rho}  V(x) u_0\phi\,dx + \int_{\R^n\setminus B_\rho} (V(x)-1) u_0\phi\,dx
      \notag
   \end{align}
  Since we want $U$ to be a weak solution of \eqref{Gl
  I} in $\R^n\setminus B_\delta$ for all $\delta>0$ we require
  \begin{align*}
    \int_{\R^n} (\nabla U \nabla \phi + V(x) U \phi)\,dx
     = \int_{\R^n} \Gamma(x) |U|^{p-1}U \phi \,dx.
  \end{align*}
  Hence, the function $\tilde{u}\in H^1(\R^n)$ that we seek must satisfy
   \begin{align}
     \int_{\R^n} (\nabla \tilde{u}\nabla\phi + V(x)\tilde{u}\phi)\,dx
     &= \int_{\R^n} \Big( \Gamma(x)|u_0+\tilde{u}|^{p-1}(u_0+\tilde u)- u_0^p\Big)\phi\,dx
     \label{Gl Equation tildeu} \\
     &\quad - \int_{B_\rho} V(x) u_0\phi\,dx  - \int_{\R^n\setminus B_\rho} (V(x)-1)
     u_0\phi\,dx  \notag \\
	&\quad - \oint_{\partial B_\rho} (\partial_{\nu}^+ u_0 -  \partial_{\nu}^- u_0)\phi\,d\sigma.
	  \notag
   \end{align}
  Thus, we will look for critical points $\tilde u \in H^1(\R^n)$ of the functional $J:H^1(\R^n)\to\R$ given by
  \begin{align}\label{Gl Def functional}
    J[u] := \frac{1}{2}\|u\|_V^2  - J_1[u] - J_2[u] + J_3[u]
  \end{align}
  where $\|\cdot\|_V$ is defined by \eqref{Gl Def innerproduct} and
  $J_i:H^1(\R^n)\to\R$ ($i=1,2,3$) are defined by
  \begin{align*}
    J_i[u] &= \int_{\R^n} F_i(u,x)\,dx, \quad i=1,2,\\
    J_3[u]
    &= \int_{B_\rho} V(x) u_0 u \,dx +  \int_{\R^n\setminus
    B_\rho} (V(x)-1) u_0u\,dx + \oint_{\partial B_\rho} (\partial_\nu^+ u_0 - \partial_\nu^-
    u_0) \gamma(u) \,d\sigma. 
  \end{align*}
  Here $\gamma:H^1(\R^n)\to L^2(\partial B_\rho)$ denotes the trace operator
  and the functions $F_1,F_2:\R\times\R^n\to\R$ are given by
  \begin{align*}
    F_1(s,x) &= \frac{1}{p+1} \big( |s+u_0(x)|^{p+1} - u_0(x)^{p+1}- (p+1) u_0(x)^p s\big), \\
    F_2(s,x) &=  \frac{\Gamma(x)-1}{p+1} \big( |u_0(x)+s|^{p+1}-u_0(x)^{p+1}\big).
  \end{align*}
  We will prove in Proposition \ref{Prop welldefinedness} that $J$ is well-defined and continuously
  Fr\'{e}chet-differentiable.
  
  \medskip 
  
  In order to find a positive distributional solution of \eqref{Gl I} in the case $\Gamma\geq 0$ we
  introduce the functional $\hat J:H^1(\R^n)\to\R$ given by
  \begin{align}
     \hat J[u] := \frac{1}{2}\|u\|_V^2  - \int_{\R^n}  \hat F_1(u,x)\,dx - \int_{\R^n}
     \hat F_2(u,x)\,dx + J_3[u]
     \label{Gl Def functional 2}
  \end{align}   
  where
  \begin{align*}
    \hat F_1(s,x) &= \frac{1}{p+1} \big( (s+u_0(x))_+^{p+1} - u_0(x)^{p+1}- (p+1) u_0(x)^ps\big),
    \\ 
    \hat F_2(s,x) &=  \frac{\Gamma(x)-1}{p+1} \big( (u_0(x)+s)_+^{p+1}-u_0(x)^{p+1}\big)
  \end{align*}
  The results of the upcoming section will hold for both $J$ and $\hat J$ due to the fact that 
  the inequalities \eqref{Gl ineq F1},\eqref{Gl ineq F2},\eqref{Gl Estimate difference quotient
  Fi},\eqref{est_fistrich} and thus \eqref{Gl Estimate F1}-\eqref{estimate_j2},\eqref{Gl
  continuityofderivative} also hold for $\hat F_1,\hat F_2$.
  
  \subsection{Existence of a critical point} \label{exist_sec}

  The proof of Theorem~\ref{Thm unbounded solution} relies on the following results. 
  First we show in Proposition~\ref{Prop welldefinedness} that the functional $J$ is well-defined
  and continuously Fr\'{e}chet-differentiable  for all $p\in(\frac{n}{n-2},\frac{n+2}{n-2})$. 
  In Proposition~\ref{Prop local structure of J} we prove next $J[u] \geq m>0$ 
  for all $u\in H^1(\R^n)$ with $\|u\|=r_0$ and all $p\in
  (\frac{n}{n-2},\frac{n}{n-2}+\eps)$  for appropriately chosen $m,r_0,\eps> 0$.  
  Using Ekeland's variational principle we then prove in
  Proposition~\ref{Prop existence critical point} the existence of a critical point $\tilde{u}$ of
  $J$. Finally, in Lemma~\ref{Lemma very weak solution} we show that $U:= u_0+\tilde{u}$
  indeed defines an unbounded distributional solution of \eqref{Gl I}.

  \medskip

  We start by proving that $J$ is well-defined and continuously Fr\'{e}chet-differentiable.

  \begin{proposition}\label{Prop welldefinedness}
    Let the assumptions of Theorem~\ref{Thm unbounded solution} hold. Then the functional $J$ given
    by \eqref{Gl Def functional} is well-defined and continuously Fr\'{e}chet-differentiable
    for all $p\in(\frac{n}{n-2},\frac{n+2}{n-2})$ with Fr\'{e}chet-derivative
    \begin{align*}
      J'[u](\phi)
      &=  \langle u,\phi \rangle_V  - \int_{\R^n} \big( F_1'(u,x) \phi +
      F_2'(u,x)\phi \big)\,dx  +  J_3[\phi].
    \end{align*}
	Here $'$ refers to the partial derivative with respect to the first variable.
  \end{proposition}
  \begin{proof}~
    {\it $J$ is well-defined:}
    First we show that $J_1,J_2$ are well-defined. The estimates
    \begin{align}
      |F_1(s,x)|
      &\leq c \big(  u_0(x)^{p-1}s^2 + |s|^{p+1} \big), \label{Gl ineq F1}\\
	  |F_2(s,x)|
      &\leq c |\Gamma(x)-1| \, \big(  u_0(x)^p|s| +  |s|^{p+1} \big)\label{Gl ineq F2}
    \end{align}
	together with \eqref{Gl_estimate} and (H3) imply
    \begin{align}
       |F_1(s,x)|
       &\leq c\,
       \begin{cases}
         |s|^{p+1}+c_{n,p}^{p-1}\frac{|s|^2}{|x|^2}, &   \text{if }  x\in B_\rho \\
         |s|^{p+1}+ c_{n,p}^{p-1}|s|^2, &   \text{if } x\in \R^n\setminus B_\rho,
       \end{cases} \label{Gl Estimate F1} \\       
        |F_2(s,x)|
       &\leq
       c\,
       \begin{cases}
          |s|^{p+1} + c_{n,p}^p |x|^{\beta-\frac{p+1}{p-1}}\,\frac{|s|}{|x|}, &
           \text{if }  x\in B_\rho \\
          |s|^{p+1} + c_{n,p}^p e^{-\frac{p}{2}|x-\rho|} |s|, &  \text{if }  x\in \R^n\setminus
          B_\rho.
       \end{cases}    \label{Gl Estimate F2}
    \end{align}
    By Hardy's inequality we obtain from \eqref{Gl Estimate F1}
    \begin{align}
    |J_1[u]| 
    &\leq c \left(\int_{\R^n} |u|^{p+1}\,dx + c_{n,p}^{p-1} \int_{B_\rho}
    \frac{|u|^2}{|x|^2}\,dx +c_{n,p}^{p-1}\int_{\R^n\setminus B_\rho} u^2\,dx\right)
    \label{estimate_j1}\\ 
    &\leq c(\|u\|^{p+1}+c_{n,p}^{p-1}\|u\|^2). \nonumber
    \end{align}
    Since $\beta>\frac{n-2}{2}$ by (H3) and $p>\frac{n}{n-2}$ we have 
    $\| |x|^{\beta-\frac{p+1}{p-1}} \|_{L^2(B_\rho)} \leq c$. Hence \eqref{Gl Estimate F2} and
    Hardy's inequality imply
    \begin{align}
      |J_2[u]| 
      &\leq c \left(\int_{\R^n} |u|^{p+1}\,dx + c_{n,p}^p \int_{B_\rho}
         |x|^{\beta-\frac{p+1}{p-1}}\frac{|u|}{|x|}\,dx+ c_{n,p}^p \int_{\R^n\setminus
        B_\rho}e^{-\frac{p}{2}|x-\rho|} |u|\,dx\right) \label{estimate_j2}\\ 
      & \leq c(\|u\|^{p+1}+c_{n,p}^p \|u\|). \nonumber
    \end{align}    
    Therefore $J_1,J_2$ are well-defined.
    
    It remains to prove that $J_3$ is
    well-defined. From $\alpha\geq \frac{n-6}{2}$ by assumption (H1) and $p>\frac{n}{n-2}$ we infer
    $|x|^{\alpha+\frac{p-3}{p-1}}\in L^2(B_\rho)$ with 
    \begin{equation}
      D(p):= \| |x|^{\alpha+\frac{p-3}{p-1}}\|_{L^2(B_\rho)}
        \leq c\,\big( 2\alpha+\frac{2p-6}{p-1}+n \big)^{-1/2}.
      \label{def_dp}
     \end{equation}
    Therefore \eqref{Gl_estimate} and Hardy's inequality yield
      \begin{align} \label{Gl estimate u_0u}
      \int_{B_\rho} |V(x) u_0 u| \,dx
      \leq c \,c_{n,p} \int_{B_\rho} |x|^{\alpha+\frac{p-3}{p-1}} \frac{|u|}{|x|}\,dx 
      \leq c\, c_{n,p} D(p)\|u\| 
      \end{align}
      so that the first integral in $J_3$ is well-defined on $H^1(\R^n)$. The remaining two integrals
     in $J_3$ are also well-defined on $H^1(\R^n)$ since $u_0$ decays exponentially at infinity and
     since the one-sided derivatives in the boundary integral exist by
     Proposition~\ref{Prop Existence approximate solution} (i),(ii). Hence, $J$ is well-defined.

    \medskip

	\noindent
	{\it Fr\'{e}chet-differentiability:} Since $J_3$ is linear we only have to deal with
	$J_1,J_2$. Similar to the calculations above we get for $i=1,2$, $x\in\R^n$, $s,t\in \R$ 
	\begin{align}\label{Gl Estimate difference quotient Fi}
      &|F_i(s+t,x)-F_i(s,x)- t F_i'(s,x)| \nonumber\\
      &\leq c \big| |u_0(x)+s+t|^{p+1} - |u_0(x)+s|^{p+1} -
        (p+1)|u_0(x)+s|^{p-1}(u_0(x)+s) t \big| \nonumber \\
      &\leq c\, \big( |u_0(x)+s|^{p-1} t^2 + |t|^{p+1} \big) \\
      &\leq c\, \big( u_0(x)^{p-1} t^2 + |s|^{p-1}t^2 + |t|^{p+1} \big) \nonumber,
     \end{align}
     where for $i=2$ we estimated $|\Gamma(x)-1|\leq \|\Gamma\|_\infty+1$. Hardy's and
     Sobolev's inequality and the exponential decay of $u_0$ from \eqref{Gl_estimate} yield      
     $$ 
       \int_{\R^n} |F_i(u+h,x)-F_i(u,x)- h  F_i'(u,x)|\,dx \leq c(\|h\|^2+ \|h\|^{2p}), \qquad i=1,2, 
     $$
	for all $u,h\in H^1(\R^n)$ which shows that the functionals $J_1, J_2$ are
	Fr\'{e}chet-differentiable.
	
  \medskip
  
    \noindent
    {\it Continuity of the Fr\'{e}chet-derivative:} Again we only need to
    consider $J_1'$ and $J_2'$. By the mean value theorem we
    get for $i=1,2$
    \begin{align}
      |F_i'(s,x)-F_i'(t,x)| 
      &\leq c\,\bigl||s+u_0(x)|^{p-1}(s+u_0(x))-|t+u_0(x)|^{p-1}(t+u_0(x)) \bigr| \nonumber\\
      & = c\,|s-t| |\sigma+u_0(x)|^{p-1}  \quad \mbox{ for $\sigma$ between $s,t$}  \label{est_fistrich}\\
      & \leq c |s-t|(|s|^{p-1}+|t|^{p-1}+ |x|^{-2}) \nonumber.
    \end{align}
    Hence, if $(u_j)$ converges to $u$ in $H^1(\R^n)$ and if $\phi\in H^1(\R^n)$ with
    $\|\phi\|=1$ then
    \begin{align} \label{Gl continuityofderivative}
    |J_i'[u_j](\phi)-J_i'[u](\phi)|
     &\leq c \int_{\R^n} \big(|u|^{p-1}+|u_j|^{p-1}+|x|^{-2}\big)|u_j-u| |\phi|\,dx
     \nonumber\\ 
     &\leq c
     (\|u\|_{L^{p+1}(\R^n)}^{p-1}+\|u_j\|_{L^{p+1}(\R^n)}^{p-1})\|u_j-u\|_{L^{p+1}(\R^n)}\|\phi\|_{L^{p+1}(\R^n)}
     \\
     & \quad + c\|u_j-u\|\|\phi\| \nonumber \\
     &\leq c
     (\|u\|^{p-1}+\|u_j\|^{p-1} +1)\|u_j-u\| \nonumber
    \end{align}
    where we have used a triple H\"older-inequality, Hardy's inequality and Sobolev's
    embedding theorem. This shows $J_i'[u_j]\to J_i'[u]$ which finishes the proof.
  \end{proof}

  \begin{remark} 
    Note that in the case $n\geq 3,\alpha<\frac{n-6}{2}$ the integral $\int_{B_\rho} V(x) 
    |x|^{-\frac{2}{p-1}} u \,dx$ is not well-defined for all $u\in H^1(\R^n)$ and
    all $p>\frac{n}{n-2}$. 
    Indeed, if $V(x)=|x|^\alpha$ near the origin and $\alpha<\frac{n-6}{2}$ then we can find 
    $p>\frac{n}{n-2}$ and $u\in H^1(\R^n)$ such that $\int_{B_\rho} |V(x)||x|^{-\frac{2}{p-1}}|u|
    \,dx=+\infty$, e.g. choose  $u(x)= |x|^{\frac{2}{p-1}-n-\alpha}e^{-|x|^2} \in H^1(\R^n)$ for
    $p\in (\frac{n}{n-2},\frac{2\alpha+n+6}{2\alpha+n+2})$ if $-\frac{n+2}{2}<\alpha<\frac{n-6}{2}$
    and $p\in (\frac{n}{n-2},\infty)$ in the case $\alpha\leq -\frac{n+2}{2}$.
  \end{remark}
  
  \begin{proposition}\label{Prop local structure of J}
    Let the assumptions of Theorem \ref{Thm unbounded solution} hold. Then there exist values
    $\eps,m,r_0>0$ such that for all  $p\in (\frac{n}{n-2},\frac{n}{n-2}+\eps)$ 
    $$
      J[u] \geq m \quad \mbox{for all } u\in H^1(\R^n) \mbox{ with } \|u\|=r_0.
    $$
  \end{proposition}
  \begin{proof}
    The choice of $\eps,m,r_0>0$ stems from the estimate
    \begin{equation}\label{Gl structure of J}
      J[u] \geq A(p) \|u\|^2 - B\|u\|^{p+1}-C(p)\|u\|
    \end{equation}
    where $A(p)\to A>0$ for some $A>0$, $B>0$ and $C(p)\to 0$ as $p\searrow \frac{n}{n-2}$.
    Let us first finish the proof assuming that \eqref{Gl structure of J} has already been
    shown.
    
    \medskip 
    
    \noindent
    {\em Choice of $\eps,m,r_0$:} Let $r_0 := \min \{
    (\frac{A}{8B})^{\frac{1}{q-1}} : \frac{n}{n-2}\leq q\leq \frac{n+2}{n-2}\}$ and $m:=
    \frac{A}{4}r_0^2$. We choose $\eps>0$ so small that for all $p\in
    (\frac{n}{n-2},\frac{n}{n-2}+\eps)$ one has $A(p)\geq \frac{A}{2}$ and $C(p)\leq
    \frac{A}{8}r_0$. Then for all $p\in (\frac{n}{n-2},\frac{n}{n-2}+\eps)$ and all $u\in
    H^1(\R^n)$ with $\|u\|=r_0$ we have 
    \begin{align*}
       A(p) \|u\|^2 - B\|u\|^{p+1}-C(p)\|u\| 
       \geq \frac{A}{2}r_0^2-Br_0^{p+1}-C(p)r_0 
       \geq r_0^2 (\frac{A}{2}-\frac{A}{8}-\frac{A}{8})
       = m
    \end{align*}
    which gives the result.
    \medskip
        
    It remains to prove \eqref{Gl structure of J}. Let $A>0$ 
    be a constant such that $\|\cdot\|^2_V \geq 2A \|\cdot\|^2$ on $H^1(\R^n)$ (see Remark
    \ref{rem I}). Using the estimates \eqref{estimate_j1}, \eqref{estimate_j2} we get 
    $$ 
    |J_1[u]| + |J_2[u]|  \leq c \,
    (\|u\|^{p+1} + c_{n,p}^{p-1}\|u\|^2+ c_{n,p}^p \|u\|). 
    $$
    From Proposition \ref{Prop Existence approximate solution}, \eqref{Gl estimate u_0u} and the
    trace theorem we obtain
 	\begin{align*}
	  |J_3[u]|
      &\leq \int_{B_\rho} |V(x) u_0u|\,dx +  \int_{\R^n\setminus B_\rho}
    			|(V(x)-1)u_0u| \,dx  + \int_{\partial B_\rho} |\partial_\nu^+ u_0 - \partial_\nu^- u_0|
    		  |\gamma(u)|\,d\sigma \\
	  &\leq c\, c_{n,p}(D(p)+1) \|u\|,
    \end{align*}
    where the value $D(p)$ is defined in \eqref{def_dp}. This results in the estimate
    \begin{align*}
      J[u]
      &\geq \frac{1}{2}\|u\|_V^2 - |J_1[u]|-|J_2[u]|-|J_3[u]| \\
      &\geq \underbrace{(A-c\, c_{n,p}^{p-1})}_{=:A(p)}\|u\|^2 -
      c\|u\|^{p+1}  - \underbrace{c\big( c_{n,p}^p + c_{n,p}(D(p)+1)\big)}_{=:C(p)}\|u\|.
    \end{align*}
    Clearly, $A(p)\to A$ as $p\searrow \frac{n}{n-2}$. Furthermore, $C(p)\to 0$ as
    $p\searrow \frac{n}{n-2}$. Indeed, if $\alpha>\frac{n-6}{2}$ then \eqref{def_dp} shows
    that $D(p)$ is uniformly bounded in $p$ for $p>\frac{n}{n-2}$. If
    $\alpha=\frac{n-6}{2}$ then $D(p)\to \infty$ but still $c_{n,p}D(p)\to 0$ as
    $p\searrow \frac{n}{n-2}$. This finally proves \eqref{Gl structure of J}.
  \end{proof}

  Now we look for a critical point within $\{ u\in H^1(\R^n) : \|u\|<r_0\}$. We recall Ekeland's
  variational principle, cf. Struwe \cite{str}, Theorem 5.1.

  \medskip

  \noindent
  {\bf Ekeland's variational principle.} {\em Let $M$ be a complete metric space with metric $d$,
  and let $J:M\to \R\cup \{+\infty\}$ be lower semi-continuous, bounded from below, and $\not \equiv
  \infty$. Then, for any $\eta, \delta>0$, and $u\in M$ with $$ J[u] \leq \inf_{M} J +\eta
  $$
  there is an element $w\in M$ strictly minimizing the functional
  $$
  J_w[z] \equiv J[z]+\frac{\eta}{\delta} d(w,z).
  $$
  Moreover, we have $J[w] \leq J[u]$ and $d(w,u)\leq \delta$.}

  \begin{proposition}\label{Prop existence critical point}
    Let the assumptions of Theorem \ref{Thm unbounded solution} hold and let $\eps,m,r_0>0$ be the
    values from Proposition~\ref{Prop local structure of J}. Then for all $p\in
    (\frac{n}{n-2},\frac{n}{n-2}+\eps)$ the functional $J$ has a nontrivial critical point
    $\tilde{u}\in H^1(\R^n)$ with $\|\tilde{u}\|\leq r_0$.
  \end{proposition}
  \begin{proof} 
    {\em Step 1:} Let us find a weakly convergent Palais-Smale sequence. Consider the minimization
    problem 
    $$
      \inf_M J \quad \mbox{ where } M=\{u\in H^1(\R^n):\|u\|\leq r_0\}.
    $$
    Choose a positive sequence
    $\eta_j\to 0$ as $j\to \infty$ and let $\tilde u_j\in M$ be such that 
    $$ 
      J[\tilde u_j] \leq \inf_M J + \eta_j^2.
    $$
    Using Ekeland's variational principle with $\eta=\eta_j^2$ and $\delta=\eta_j$ we find $u_j\in M$
    such that
    $$
      J[u_j] \leq J[z]+ \eta_j \|z-u_j\| \quad \mbox{ for all } z \in M.
    $$
    Then $(u_j)$ is also a minimizing sequence for $J|_M$ and since
    $0\in M$ and $J[0]=0<m$ we see that  $\|u_j\|<r_0$ for large $j$. Hence, almost all
    $u_j$ are interior points of $M$. Applying the estimate
    \begin{align*}
      J[z] &= J[u_j]+ J'[u_j](z-u_j)+ o(\|z-u_j\|) \\
      &\leq J[z]+ J'[u_j](z-u_j)+\eta_j\|z-u_j\|+o(\|z-u_j\|) \quad \mbox{ as } z\to u_j, z \in M
    \end{align*}
    to $z=u_j+tv$ with $\|v\|=1$ we find for $t\to 0$
    $$
      \|J'[u_j]\|= \sup_{\|v\|=1} |J'[u_j](v)| \leq \eta_j\to 0 \quad \mbox{ as } j \to \infty,
    $$
    i.e. $(u_j)$ is a minimizing Palais-Smale sequence of $J|_M$. Moreover, since
    $(u_j)$ is bounded in $H^1(\R^n)$ by $r_0$ we may assume (up to selecting
    subsequences) that $u_j\rightharpoonup \tilde u$ in $H^1(\R^n)$ and $u_j\to \tilde u$ almost
    everywhere in $\R^n$.
    
    \medskip

    \noindent
    {\em Step 2:} Let us show that the weak limit $\tilde u$ is a critical point of $J$.
    So let $\phi\in C_0^\infty(\R^n)$ be a fixed test function, $K:=\supp(\phi)$. Because of
    $u_j\to \tilde u$ in $L^{p+1}(K)$ by compact embedding we may use Lemma A.1 in \cite{wil} 
    to find a function $w_\phi\in L^{p+1}(K)$ and a subsequence (possibly depending on $\phi$) again
    denoted by $(u_j)$ such that $|\tilde u|,|u_j|\leq w_\phi$. Recalling \eqref{est_fistrich} we get 
    $$
      |J_i'[u_j]\phi-J_i'[\tilde u]\phi|
       \leq c \int_K \big(w_\phi^{p-1}+\frac{1}{|x|^2}\big)|u_j-\tilde u| |\phi|\,dx
       \quad \mbox{ for } i=1,2 \mbox{ and }j\in\N.
    $$
    The integrand is pointwise almost everywhere bounded by
    $2w_\phi^p|\phi|+\frac{2}{|x|^2}w_\phi|\phi|$. Since $w_\phi\in L^{p+1}(K),\phi\in L^\infty(K)$
    and $|x|^{-2}\in L^\frac{p+1}{p}(K)$ the dominated convergence theorem applies and yields 
    $$
      J_i'[u_j](\phi) \to J_i'[\tilde u](\phi) \quad \mbox{ for } i=1,2 \mbox{ as } j\to \infty.
    $$
    Weak convergence implies $\langle u_j,\phi \rangle_V\to \langle \tilde u,\phi \rangle_V$.
    Furthermore $J_3'[u_j](\phi)=J_3'[\tilde u](\phi)=J_3[\phi]$ by linearity. In total we see that
    $J'[\tilde u](\phi)=\lim_{j\to \infty} J'[u_j](\phi) = 0$ for every $\phi\in
    C_0^\infty(\R^n)$ which proves the result.
  \end{proof}
  
\subsection{The distributional solution property} \label{distrib_sec}
  
  In Proposition~\ref{Prop existence critical point} we have proved that under the assumptions of
  Theorem~\ref{Thm unbounded solution}  a critical point
  $\tilde u\in H^1(\R^n)$ of $J$ exists provided $\eps>0$ is sufficiently small. Due to
  the properties of $u_0$ (cf. Proposition~\ref{Prop Existence approximate solution}) we find that
  $U=u_0+\tilde u$ lies in $H^1(\R^n\setminus B_\delta)$ for every $\delta>0$ and $U\in
  L_{loc}^q(\R^n)$ for all $q\in [1,\frac{n(p-1)}{2})$. From part (iii) of Theorem \ref{Thm
  unbounded solution} which is proved in the next section we get $U\in L^q(\R^n)$ for all $q\in
  [1,\frac{n(p-1)}{2})$. Since the Euler-equation
  \eqref{Gl Equation tildeu} for $\tilde u$ and equation \eqref{Gl Equation u0} hold for all
  $\varphi\in C_0^\infty(\R^n\setminus \{0\})$ we obtain that for every $\delta>0$ the function
  $U=u_0+\tilde u$ is a weak solution of \eqref{Gl I} on $\R^n\setminus B_\delta$. 
  
  \medskip 
  
  In order to complete the proof of Theorem~\ref{Thm unbounded solution},(i),(ii) it therefore
  remains to show that $U$ is an unbounded distributional solution of \eqref{Gl I}.

 \begin{lemma} \label{Lemma very weak solution}
  Let the assumptions of Theorem~\ref{Thm unbounded solution} hold and let $\tilde u\in H^1(\R^n)$
  be a critical point of $J$ according to Proposition~\ref{Prop existence critical point}. Then the
  function $U:= u_0+\tilde u$ is a distributional solution of \eqref{Gl I} with $\esssup_{B_\delta} U = +\infty$ for all $\delta>0$.  
\end{lemma}
\begin{proof} 
  According to the definition of $u_0$ for all $\delta>0$:
  \begin{align*}
    \int_{B_\delta} |u_0(x)|\,dx 
    &= O(\delta^{-\frac{2}{p-1}+n}), & \int_{B_\delta} |u_0(x)|^p\,dx   &=
    O(\delta^{-\frac{2p}{p-1}+n}),\\ \oint_{\partial B_\delta} |u_0(x)|\,dx
    &= O(\delta^{-\frac{2}{p-1}+n-1}), &  \oint_{\partial B_\delta} |\partial_\nu^{\pm}
    u_0(x)|\,dx &= O(\delta^{-\frac{p+1}{p-1}+n-1}).
  \end{align*}
  All integrals converge to $0$ as $\delta\to 0$ since
  $p>\frac{n}{n-2}>\frac{n+1}{n-1}>\frac{n+2}{n}$. Hence, for all $\phi\in C_0^\infty(\R^n)$ we
  find from Proposition \ref{Prop Existence approximate solution},(i)
  \begin{align} \label{Gl VWS I}
   \int_{B_\rho} u_0 (-\Delta \phi)\,dx
   &= \lim_{\delta\to 0} \int_{B_\rho\setminus B_\delta} u_0 (-\Delta\phi)\,dx \nonumber \\
   &= \lim_{\delta\to 0} \int_{B_\rho\setminus B_\delta} (-\Delta u_0) \phi\,dx -
     \oint_{\partial B_\rho} (u_0\partial_\nu^+\phi -\phi\partial_\nu^+ u_0 ) \,d\sigma \nonumber\\
   &= \int_{B_\rho} u_0^p \phi\,dx -
     \oint_{\partial B_\rho} (u_0\partial_\nu^+\phi -\phi\partial_\nu^+ u_0 ) \,d\sigma 
   \end{align}
   and since $\phi$ has compact support Proposition \ref{Prop Existence approximate solution},(ii)
   implies
   \begin{align} \label{Gl VWS II}
   \int_{\R^n\setminus B_\rho} u_0 (-\Delta \phi)\,dx
   &= \int_{\R^n\setminus B_\rho} (-\Delta u_0) \phi\,dx
    + \oint_{\partial B_\rho} (u_0\partial_\nu^-\phi-\phi\partial_\nu^- u_0) \,d\sigma \nonumber \\
   &= \int_{\R^n\setminus B_\rho} (u_0^p-u_0) \phi\,dx
    + \oint_{\partial B_\rho} (u_0\partial_\nu^-\phi-\phi\partial_\nu^- u_0) \,d\sigma.
   \end{align}
   Since $\phi$ is smooth we have $\partial_\nu^-\phi=\partial_\nu^+\phi$ on $\partial B_\rho$.
   Using (H1) we find $Vu_0\in L^1_{loc}(\R^n)$ by direct calculation. Hence,
   \begin{align}\label{Gl VWS III}
    \int_{\R^n} u_0 (-\Delta\phi + V(x)\phi)\,dx
    &= \int_{\R^n} u_0^p \phi\,dx + \int_{B_\rho} V(x) u_0\phi\,dx
         \notag  \\
	    &\quad + \int_{\R^n\setminus B_\rho} (V(x)-1)  u_0\phi\,d\sigma  + \oint_{\partial B_\rho}
    (\partial_{\nu}^+ u_0 -\partial_{\nu}^- u_0) \phi\,d\sigma.
  \end{align}
  On the other hand $\tilde{u}$ is a critical point of $J$ and thus satisfies  
  the Euler-equation \eqref{Gl Equation tildeu} for all $\phi\in H^1(\R^n)$. Moreover, 
  $V\tilde u\in L^1_{loc}(\R^n)$ and hence, $VU=Vu_0+V\tilde{u}\in L^1_{loc}(\R^n)$. Adding up
  \eqref{Gl Equation tildeu} and \eqref{Gl VWS III} gives
  \begin{align*}
     \int_{\R^n} U(-\Delta \phi + V(x) \phi)\,dx
     &= \int_{\R^n} \Gamma(x) |U|^{p-1}U\phi\,dx\quad\mbox{for all } \phi\in C_0^\infty(\R^n).
   \end{align*}
   Hence, $U$ is a distributional solution of \eqref{Gl I}.
   
   \medskip
   
   Now assume $U\leq C_\delta < \infty$ almost everywhere on $B_\delta$ for some $\delta>0$.
   Choosing $\delta'\in (0,\delta)$ such that $u_0(x)\geq 2C_\delta$ on $B_{\delta'}$ (see
   Proposition \ref{Prop Existence approximate solution},(iv)) we get $\tilde u = U-u_0 \leq
   -\frac{u_0}{2}<0$ almost everywhere on $B_{\delta'}$ and thus 
    $$
      \|\tilde u\|_{L^\frac{2n}{n-2}(B_{\delta'})} \geq \frac{1}{2}\|u_0\|_{L^\frac{2n}{n-2}(B_{\delta'})}= +\infty
    $$
    which contradicts $\tilde u\in H^1(\R^n)$.
    Hence, $\esssup_{B_\delta} U = +\infty$. 
 \end{proof}

  \begin{remark}
    Clearly, $u_0\notin H^1(B_1)$ so that $U:= u_0+\tilde{u}\notin H^1(\R^n)$.
  \end{remark}

\subsection{Exponential decay} \label{decay_sec}
   
  Let us prove part (iii) of Theorem \ref{Thm unbounded solution}. For the reader's convenience we
  only present the main idea of the proof, details are given in Appendix B.
    
  \begin{lemma}\label{Lemma exponential decay}
    Let the assumptions of Theorem~\ref{Thm unbounded solution} hold and let $\tilde u\in H^1(\R^n)$
    be a critical point of $J$ according to Proposition~\ref{Prop existence critical point}, let
    $U:= u_0+\tilde u$. Then for all $0<\mu<\sqrt{\Sigma}$ there is $C_\mu>0$ such that $|U(x)|\leq
    C_\mu e^{-\mu|x|}$ for all $x\in\R^n$ with $|x|\geq 1$. 
  \end{lemma}
  \begin{proof}
    Applying Proposition~\ref{Prop decayToZero} to $u=U$, $\Omega=\R^n\setminus B_2,q=p$ and $W:=
    V-\Gamma |U|^{p-1}1_{\R^n\setminus B_2}$ we deduce that $U$ can 
    be assumed to be continuous and that we have $U(x)\to 0$ as $|x|\to\infty$. Note that
    $W\in L^\infty(\R^n\setminus B_1)+L^{\frac{2n}{(n-2)(p-1)}}(\R^n\setminus B_1)\subset
    K_n(\R^n\setminus B_2)$ due to $\frac{2n}{(n-2)(p-1)}>\frac{n}{2}$ and \eqref{Gl KatoNorm}. 
    From $U(x)\to 0$ as $|x|\to\infty$ and \cite{pankov2}, Theorem 8.3.1 we obtain 
    $$
      \sigma_{ess}(-\Delta+W) = \sigma_{ess}(-\Delta+V)\subset [\Sigma,\infty).
    $$
    Then Proposition~\ref{Prop exponentialdecay} applied to $\Omega=\R^n\setminus B_2$,
    $s=\frac{2n}{(n-2)(p-1)},q=2$ gives $|U(x)| \leq C_\mu' e^{-\mu|x|}$ for all $x\in \R^n$ with
    $|x|\geq 3$. Since $U\in H^1(\R^n\setminus B_\delta)$ satisfies a subcritical elliptic PDE in $\R^n\setminus B_\delta$
    for all $\delta>0$ the result follows from the DeGiorgi-Nash-Moser local boundedness principle.
  \end{proof} 

\subsection{Positivity in the case $\Gamma\geq 0$} \label{pos_sec}
  
  In this section we prove part (iv) of Theorem \ref{Thm unbounded solution}, so let us assume $\Gamma\geq 0$. 
  As pointed out before (see \eqref{Gl Def functional 2}
  and the following remarks) the results of the previous sections \ref{exist_sec},
  \ref{distrib_sec}, \ref{decay_sec} also apply to $\hat J$, in particular we find a critical point
  $\hat u$ of $\hat J$. By Lemma~\ref{Lemma very weak solution} the function $\hat U=u_0+\hat u$
  satisfies $\esssup_{B_\delta} \hat U = +\infty$ for all $\delta>0$ and is a distributional solution of 
  $$
    \int_{\R^n} \hat U(-\Delta \phi + V(x) \phi)\,dx
     = \int_{\R^n} \Gamma(x)  {\hat U}_+^p \phi\,dx
     \quad\mbox{for all } \phi\in  C_0^\infty(\R^n). 
   $$
  It remains to show that $\hat U$ must be positive.
  
  \medskip     
    
  To this end let $\psi\in C_0^\infty(\R^n), \psi\geq 0$ be arbitrary, set $K:=\supp(\psi)$.
   Let then $w\in H^1(\R^n)$ be the unique weak solution of $-\Delta w +
   V(x) w = \psi$ obtained by minimizing the functional $L[z] := \int_{\R^n} |\nabla z|^2+V(x)z^2-2\psi z
   \,dx$ over $H^1(\R^n)$. 
   
   \medskip
 
   Since $\psi\geq 0$ one sees that $w\geq 0$ (if $w$ is a minimizer then also $|w|$ is a minimizer
   and $L$ has a unique minimizer). Then $-\Delta w=f$ in the weak sense where $f=\psi-Vw$ and
   $V\in L^q_{loc}(\R^n)$ for all $q\in [1,\frac{2n}{(6-n)_+})$. From (H1) and $w\in
   L^{\frac{2n}{n-2}}(\R^n)$ we infer $f\in L^q_{loc}(\R^n)$ for all $q\in [1,\frac{n}{2})$ so that
   Cald\'{e}ron-Zygmund estimates (cf. Chapter 9 in \cite{giltru}) imply $w\in W^{2,q}_{loc}(\R^n)$
   for all $q\in [1,\frac{n}{2})$. Sobolev's imbedding theorem then implies $f\in L^q_{loc}(\R^n)$ for all $q\in
   [1,\frac{2n}{(6-n)_+})$ and thus $w\in W^{2,q}_{loc}(\R^n)$ for all $q\in
   [1,\frac{2n}{(6-n)_+})$ again by Cald\'{e}ron-Zygmund estimates. In particular, up to a set of
   measure zero $w$ is locally uniformly continuous and satisfies $-\Delta w + Vw = \psi$ pointwise
   in $\R^n$.
   
   \smallskip
   
   Since $p>\frac{n}{n-2}$ we can find $s\in (\frac{n(p-1)}{n(p-1)-2},\frac{2n}{(6-n)_+})$. Recall
   from Section~\ref{distrib_sec} that this choice of $s$ implies $\hat U \in
   L^{\frac{s}{s-1}}(K)$. Let $(\phi_k)$ be a sequence of positive $C_0^\infty(\R^n)$-functions
   such that $\phi_k\to w$ uniformly on $K$ and in $W^{2,s}(K)$. Then $\hat U V\in L^1(K)$ and
   \begin{align*}
     \int_{\R^n} \hat U(x)\psi(x)\,dx     
     &= \int_K \hat U(x) \big( -\Delta w + V(x) w \big)\,dx \\
     &= \lim_{k\to\infty} \int_K \hat U(x) \big( -\Delta \phi_k + V(x) \phi_k \big)\,dx \\ 
     &= \lim_{k\to\infty} \int_K \Gamma(x)\hat U(x)_+^p\phi_k(x)\,dx \\
     &= \int_K \Gamma(x)\hat U(x)_+^p w(x)\,dx 
     \geq  0.
   \end{align*}   
   Since $\psi\in C_0^\infty(\R^n),\psi\geq 0$ is arbitrary we obtain $\hat U\geq 0$ almost everywhere. \qed


\section{Proof of Theorem \ref{Thm regularity}}

  Under the assumptions of Theorem~\ref{Thm regularity} we now prove regularity
  properties of distributional solutions of \eqref{Gl II} in the case $1<p<\frac{n}{(n-2)_+}$. For
  $\omega>0$ we rewrite \eqref{Gl II} in the following way
  \begin{align}\label{Gl IInew}
    -\Delta u + \omega u &= g_\omega \quad\mbox{ where }
    g_\omega(x):= g(x,u(x)) + (\omega-V(x))u(x).
  \end{align}
  We will show that \eqref{Gl IInew} can be written in form of an integral equation using the
  Green function $G_\omega$ of $-\Delta + \omega$. Therefore we are lead to study the operator
  $T_\omega$ given by
  \begin{align*}
    T_\omega(f) 
    := \int_{\R^n} G_\omega(x-y)f(y)\,dy.  
  \end{align*}
  It is well-known (cf. \cite{graryz}, \cite{ste}) that
  \begin{align}\label{Gl Def GreensF G}
    G_\omega(x)
    = \omega^{\frac{n-2}{2}} G_1(\sqrt{\omega} x)
    = (2\pi)^{-\frac{n}{2}} |\omega^{-1/2} x|^{\frac{2-n}{2}}K_{\frac{n-2}{2}}(\sqrt{\omega}|x|).
  \end{align} 
  The following expansions can be found in \cite{graryz} for multiindices $\alpha$ with
  $|\alpha|\geq 1$:
  \begin{align}
    G_\omega(x)  =
    \begin{cases}
      O(1), & n=1 \\
      O(\log\frac{1}{|x|}), & n=2 \\
      O(|x|^{2-n}), & n\geq 3
    \end{cases}
    \quad\mbox{ as } |x|\to 0,\quad 
     G_\omega(x) = O(e^{-\sqrt{\omega}|x|}) \mbox{ as } |x|\to\infty. 
    \label{Gl GreensF asymptotic I}  \\ 
      D^\alpha G_\omega(x) = O(|x|^{2-n-|\alpha|}) \quad \mbox{ as } |x|\to 0, \quad   
      D^\alpha G_\omega(x) = O(e^{-\sqrt{\omega}|x|}) \mbox{ as }
      |x|\to\infty. \label{Gl GreensF asymptotic II}     
  \end{align}
  
  \medskip
  
  The proof of Theorem \ref{Thm regularity} is given in three steps: In Proposition~\ref{Prop
  Mapping properties} and Proposition~\ref{Prop local Mapping properties} we study the mapping
  properties of $T_\omega$ for fixed $\omega>0$ in order to prove in Proposition~\ref{Prop
  representation formula} the representation formula $u = T_\omega(g_\omega)$ for every
  distributional solution $u$ of \eqref{Gl II} with $u\in L^p(\R^n;\omega_0)$ and $\omega_0<\omega$.
  Finally we obtain the regularity result of Theorem~\ref{Thm regularity} by 
  a combination of the mapping properties of $T_\omega$ with the continuity/decay results of
  Proposition~\ref{Prop decayToZero} and Proposition~\ref{Prop exponentialdecay}.
    
  \begin{proposition}\label{Prop Mapping properties}
    Let $\omega>0$, $k\in \{0,1,2\}$ and $q,r\in [1,\infty]$. Then 
    $$ 
    T_\omega: L^q(\R^n)\to W^{k,r}(\R^n)
    $$
    provided $s:=(1+\frac{1}{r}-\frac{1}{q})^{-1}$ satisfies one of the following
    conditions:
    \begin{itemize}
      \item[(i)] Case $k=0$: $s\in [1,\frac{n}{(n-2)_+})$ or
      $n=1,s=\infty$ or $n\geq 3,q\in (1,\frac{n}{2}),s=\frac{n}{n-2}$.
      \item[(ii)] Case $k=1$: $s\in [1,\frac{n}{(n-1)_+})$ or
      $n=1,s=\infty$ or  $n\geq 2,q\in (1,n),s=\frac{n}{n-1}$.
      \item[(iii)] Case $k=2$: $q=r \in (1,\infty)$.
    \end{itemize}
    In each case there exists a constant $c=c(k,q,r,n)>0$ such that
    $$
      \|T_\omega f\|_{W^{k,r}(\R^n)} \leq c \|f\|_{L^q(\R^n)}\quad \mbox{ for all }f\in
      L^q(\R^n).
    $$ 
    Furthermore, in the cases $k=1$ or $k=2$ we have for all $|\alpha|=1$
    $$
      D^\alpha (T_\omega f)(x)= \int_{\R^n} (D^\alpha G_\omega)(x-y) f(y)\,dy.
    $$
  \end{proposition}	  
  

  \begin{proof}  
     The proof of (iii) can be found in \cite{ste}, Theorem 3, Chapter V. Let us prove (i), i.e, $k=0$. 
     Young's inequality gives 
     $$ 
          \|T_\omega f\|_{L^r(\R^n)}           
          = \|G_\omega\ast f\|_{L^r(\R^n)} 
          \leq  \|G_\omega\|_{L^s(\R^n)} \|f\|_{L^q(\R^n)}
      $$ 
    provided $q,r,s \in [1,\infty]$ satisfy $1+\frac{1}{r}= \frac{1}{s}+\frac{1}{q}$. In the cases
    $n=1,n\geq 2$ the asymptotic formulas \eqref{Gl GreensF asymptotic I}, \eqref{Gl GreensF
    asymptotic II} show that $G_\omega\in L^s(\R^n)$ for all $s\in [1,\infty], [1,\frac{n}{n-2})$
    respectively and the first two subcases are proved. The case $n\geq 3,q\in
    (1,\frac{n}{2}),s=\frac{n}{n-2}$ follows from (iii) and Sobolev's imbedding theorem $W^{2,q}(\R^n)\to 
    L^{\frac{nq}{n-2q}}(\R^n)$.
   
    \medskip
 
    \noindent  Next we prove (ii). By \eqref{Gl GreensF asymptotic I}, \eqref{Gl
    GreensF asymptotic II} we have $|\nabla G_\omega(z)|\sim |z|^{1-n}$ as $z\to 0$ and
    $|\nabla G_\omega(z)|\sim e^{-\sqrt{\omega}|z|}$ as $|z|\to \infty$. Hence $|\nabla
    G_\omega| \in L^s(\R^n)$ for $s\in [1,\infty],[1,\frac{n}{n-1})$ in the cases
    $n=1,n\geq 2$ respectively. In these cases the dominated convergence theorem and Young's
    inequality apply and yield $\nabla (T_\omega f) = \nabla G_\omega \ast f$ as well as 
    $$ 
        \| |\nabla (T_\omega f)|\|_{L^r(\R^n)} 
        \leq \| |\nabla G_\omega| \|_{L^s(\R^n)} \|f\|_{L^q(\R^n)}. 
    $$ 
    The case $n\geq 2,q\in (1,n),s=\frac{n}{n-1}$
    again follows from the case $k=2$ and Sobolev's imbedding theorem $W^{2,q}(\R^n)\to
    W^{1,\frac{nq}{n-q}}(\R^n)$.
\end{proof}

  We will also need the following local version of Proposition~\ref{Prop Mapping properties} where we
  use weighted Lebesgue spaces
  \begin{align} \label{Gl Def WeightedLpspace}
     L^q(\R^n;\omega) := \big\{ u\in L^q_{loc}(\R^n)\;:\; \int_{\R^n} |u(x)|^q
    e^{-\sqrt{\omega}|x|} \,dx < \infty \big\}. 
  \end{align}
   for $1\leq q<\infty$ and $\omega> 0$. We set $\|u\|_{L^q(\R^n;\omega)} := \left(\int_{\R^n}
   |u(x)|^q e^{-\sqrt{\omega}|x|} \,dx\right)^{1/q}$.
    
  \begin{proposition}\label{Prop local Mapping properties}
     Let $\omega>0$, $k\in \{0,1\}$ and $q,r\in [1,\infty]$. Then 
    $$ 
      T_\omega: L^q_{loc}(\R^n)\cap L^1(\R^n;\omega) \to W^{k,r}_{loc}(\R^n)
    $$
    provided $s:=(1+\frac{1}{r}-\frac{1}{q})^{-1}$ satisfies one of the following
    conditions:
    \begin{itemize}
      \item[(i)] Case $k=0$: $s\in [1,\frac{n}{(n-2)_+})$ or
      $n=1,s=\infty$.
      \item[(ii)] Case $k=1$: $s\in [1,\frac{n}{(n-1)_+})$ or
      $n=1,s=\infty$.
    \end{itemize}    
    In each case for all compact sets $K_1,K_2$ such that $K_1\subset\subset K_2$ there
    exists a constant $c=c(k,q,r,n,K_1,K_2)>0$ such that 
    $$
      \|T_\omega f\|_{W^{k,r}(K_1)} \leq c \big(\|f\|_{L^q(K_2)}+\|f\|_{L^1(\R^n;\omega)}\big)
      \quad\mbox{ for all } f\in L^q_{loc}(\R^n)\cap L^1(\R^n;\omega)
    $$
    First order derivatives of $T_\omega f$ can be taken under the integral as in
    Proposition~\ref{Prop Mapping properties}.
    \end{proposition}

  \begin{proof}
    Consider first the case $k=0$. For given compact sets $K_1,K_2$ with $K_1\subset\subset
    K_2$ let $B$ be an open ball centered at $0$ such that $K_1+\overline{B}\subset
    K_2$ where $K_1+\overline{B}$ denotes the Minkowski sum of $K_1$ and $\overline{B}$. Then there
    exists $C_B>0$ such that $|G_\omega(z)|\leq C_B e^{-\sqrt{\omega}|z|}$ for all $z\in
    \R^n\setminus B$, cf. \eqref{Gl GreensF asymptotic I}. If $q,r$ are as in the theorem with
    $r<\infty$ then Proposition \ref{Prop Mapping properties},(i) shows
      \begin{align*}
  	    \|T_\omega f\|_{L^r(K_1)}^r
      	&\leq \int_{K_1} \left( \int_{x+B} G_\omega(x-y)|f(y)|\,dy +
       	\int_{x+\R^n\setminus B} G_\omega(x-y)|f(y)|\,dy\right)^r \,dx \\
      	&\leq c\left( \|T_\omega (|f\,1_{K_2}|)\|_{L^r(\R^n)}^r  +
	        C_B^r \int_{K_1} \big(\int_{\R^n} e^{\sqrt{\omega}(|x|-|y|)}|f(y)|\,dy\big)^r \,dx\right) 
	        \\ 
	    &\leq c\left(  \|f\,1_{K_2}\|_{L^q(\R^n)}^r  
	        + \big(\int_{\R^n} e^{-\sqrt{\omega}|y|}|f(y)|\,dy\big)^r \right) \\
      	&= c( \|f\|_{L^q(K_2)}^r  + \|f\|_{L^1(\R^n;\omega)}^r	).
      \end{align*} 
    In the case $r=\infty$ we obtain with the
    same notations as above
    \begin{align*}
       \|T_\omega f\|_{L^\infty(K_1)}  
      &\leq c\left(\| T_\omega (|f\,1_{K_2}|) \|_{L^\infty(\R^n)}
        + \| \int_{\R^n} e^{\sqrt{\omega}(|\cdot|-|y|)} |f(y)| \,dy\|_{L^\infty(K_1)} \right)  \\ 
      &\leq c \left( \|f\|_{L^q(K_2)} + \|f\|_{L^1(\R^n;\omega)} \right).
    \end{align*}
    This finishes the proof of (i). The case $k=1$ is treated similarly using
    the mapping property (ii) in Proposition~\ref{Prop Mapping properties} instead of (i).  
  \end{proof}
     
  Next we prove the representation formula $u=T_\omega(g_\omega)$ for distributional solutions
  $u$ of \eqref{Gl II} which satisfy $u\in L^p(\R^n;\omega_0)$ for some $\omega_0<\omega$. To
  this end we first show that the corresponding linear problem has at most one solution in
  $L^1(\R^n;\omega)$.
  
  \begin{proposition}\label{Prop uniqueness linear}
    Let $v\in L^1(\R^n;\omega)$ be a distributional solution of $-\Delta v + \omega v = 0$. Then
    $v=0$.  
  \end{proposition}
  \begin{proof}
     Let $\psi\in C_0^\infty(\R^n)$ be  arbitrary and for $R>0$ set $\phi_R := \chi_R
     T_\omega(\psi)\in C_0^\infty(\R^n)$ where $\chi_R(x)=\chi(R^{-1}x)$ for a fixed function $\chi\in
     C_0^\infty(\R^n)$ with $\chi(0)=1$. Since $v\in L^1(\R^n;\omega)$ we have
     $|T_\omega(\psi)||v|+|\nabla T_\omega(\psi)||v|\in
     L^1(\R^n)$. Hence the dominated convergence theorem gives
     \begin{align*}
       0 
       &= \lim_{R\to\infty} \int_{\R^n} v(-\Delta\phi_R + \omega\phi_R)\,dx \\
       &= \lim_{R\to\infty} \Big[\int_{\R^n}  \chi_R  v\psi \,dx
          + \int_{\R^n} \big(-\Delta\chi_R T_\omega(\psi) -2\nabla\chi_R \nabla
       T_\omega(\psi) \big) v \,dx \Big] \\
       &=  \int_{\R^n}  v\psi \,dx.       
     \end{align*}
     Since $\psi\in C_0^\infty(\R^n)$ was arbitrary we get $v=0$.  
  \end{proof}

  \begin{proposition}\label{Prop representation formula}
    Let $1\leq p <\infty$ and let $g:\R^n\times\R\to \R$ be a Carath\'{e}odory function with $|g(x,s)|\leq C(1+|s|^p)$ for all $s\in\R$ and almost all $x\in\R^n$. Let $u\in L^p(\R^n;\omega_0)$ for some $\omega_0>0$ be a distributional solution of \eqref{Gl II}. Then for all $\omega>\omega_0$ we have $u=T_\omega(g_\omega)$ almost everywhere on $\R^n$
    with $g_\omega$ given by \eqref{Gl IInew}.
  \end{proposition}
  \begin{proof}
    By assumption the function $u\in L^p(\R^n;\omega_0)\subset L^1(\R^n;\omega)$ satisfies
    \begin{align*}
      \int_{\R^n} u(-\Delta\phi + \omega\phi)\,dx = \int_{\R^n} g_\omega \phi\,dx
      \qquad\forall \phi\in C_0^\infty(\R^n).
    \end{align*}
    On the other hand let us show that $T_\omega(g_\omega)\in L^1(\R^n;\omega)$ satisfies the same integral relation.
    Indeed, we have $g_\omega = g(\cdot,u) + (\omega-V)u \in L^1(\R^n;\omega_0)$ so that
    \eqref{Gl GreensF asymptotic I} implies
    \begin{align*}
      \int_{\R^n} |T_\omega(g_\omega)| e^{-\sqrt{\omega}|x|} \,dx
      &\leq \int_{\R^n} \int_{\R^n} G_\omega(x-y)|g_\omega(y)| e^{-\sqrt{\omega}|x|} \,dx \,dy\\
      &= \int_{\R^n} |g_\omega(y)|e^{-\sqrt{\omega_0}|y|}  \int_{\R^n}
      e^{\sqrt{\omega_0}|y|} G_\omega(x-y) e^{-\sqrt{\omega}|x|} \,dx \,dy\\
      &\leq \int_{\R^n} |g_\omega(y)|e^{-\sqrt{\omega_0}|y|}  \Big[ 
      c\int_{\{|x-y|\geq 1\}}  e^{\sqrt{\omega_0}|y|}e^{-\sqrt{\omega}|x-y|} e^{-\sqrt{\omega}|x|}
      \,dx \\
      &\qquad + \int_{\{|x-y|\leq 1\}} e^{\sqrt{\omega_0}|y|} G_\omega(x-y) e^{-\sqrt{\omega}|x|} \,dx
       \Big]\,dy  \\
      &\leq \int_{\R^n} |g_\omega(y)|e^{-\sqrt{\omega_0}|y|}  \Big[ c 
      \int_{\{|x-y|\geq 1\}}  e^{\sqrt{\omega_0}|y|}e^{-\sqrt{\omega_0}|x-y|} e^{-\sqrt{\omega}|x|} \,dx   \\
      &\qquad
      + \int_{\{|x-y|\leq 1\}} e^{\sqrt{\omega_0}|y|} G_\omega(x-y) e^{-\sqrt{\omega}(|y|-1)} \,dx
       \Big] \,dy \\
      &\leq c \int_{\R^n} |g_\omega(y)|e^{-\sqrt{\omega_0}|y|}  \Big[ 
      \int_{\{|x-y|\geq 1\}} e^{(\sqrt{\omega_0}-\sqrt{\omega})|x|} \,dx  
      + \int_{\{|z|\leq 1\}} G_\omega(z) \,dz \Big] \,dy \\
      &\leq c \int_{\R^n} |g_\omega(y)|e^{-\sqrt{\omega_0}|y|}\,dy
      < \infty,
    \end{align*}
    where we have used that $G_\omega$ is a locally integrable function. Furthermore, Fubini's theorem yields for $\phi \in
    C_0^\infty(\R^n)$
    \begin{align}
      \int_{\R^n} T_\omega(g_\omega)(-\Delta\phi+\omega\phi)\,dx
      &= \int_{\R^n} \left( \int_{\R^n} G_\omega(x-y)g_\omega(y)  \,dy\right)
      (-\Delta\phi(x)+ \omega\phi(x)) \,dx \notag  \\ 
      &= \int_{\R^n} g_\omega(y)
      \left(\int_{\R^n}G_\omega(x-y)(-\Delta\phi(x)+ \omega\phi(x))\,dx\right)\,dy
      \notag\\
      &= \int_{\R^n} g_\omega(y)
      \left(\int_{\R^n}G_\omega(y-x)(-\Delta\phi(x)+ \omega\phi(x))\,dx\right)\,dy
      \notag\\
      &= \int_{\R^n} g_\omega(y)  \phi(y)\,dy.
    \end{align}
     Applying Proposition~\ref{Prop uniqueness linear} to $v=u-T_\omega(g_\omega)$ we conclude $u = T_\omega(g_\omega)$. 
   \end{proof}~

   \subsection{Proof of Theorem~\ref{Thm regularity},(2)} 
      
   Let $g$ satisfy \eqref{Gl growth cond II} and let $u\in L^p(\R^n) $ be a distributional solution
   of \eqref{Gl II}. Then \eqref{Gl growth cond II} and the assumption $1<p<\frac{n}{n-2}$ implies that
   $$
     W(x):= V(x)- \frac{g(x,u(x))}{u(x)}1_{\{u(x)\neq 0\}}
   $$ 
   lies in the Kato class $K_n$ ~-- see \eqref{Gl KatoNorm} ~-- and thus Proposition \ref{Prop
   decayToZero} implies $u\in L^\infty(\R^n)$ and $u(x)\to
   0$ as $|x|\to\infty$. Hence, $u\in L^p(\R^n)\cap L^\infty(\R^n)$ and thus $g_\omega\in
   L^p(\R^n)\cap L^\infty(\R^n)$ where $g_\omega$ is defined in \eqref{Gl IInew}. From
   Proposition~\ref{Prop representation formula} we get $u=T_\omega(g_\omega)$. From
   Proposition~\ref{Prop Mapping properties} with $(k,q,r)=(1,q,q), q\in [p,\infty]$ and
   $(k,q,r)=(2,q',q'), q'\in [p,\infty)$ we get $u\in W^{1,q}(\R^n)\cap W^{2,q'}(\R^n)$ for all
   $q\in [p,\infty],q'\in [p,\infty)$.
   
   \medskip
   
   Now, in addition let us assume (H2) and \eqref{Gl growth cond III}. Then \cite{pankov2}, Theorem 8.3.1 implies 
   $$
     \sigma_{ess}(-\Delta + W(x)) = \sigma_{ess}(-\Delta + V(x)) \subset [\Sigma,\infty) 
   $$
   Hence, Proposition~\ref{Prop exponentialdecay} applies to $u$ and $\Omega=\R^n$ and it follows
   $|u(x)| \leq C_\mu e^{-\mu|x|}$ for almost all $x\in\R^n$. In  particular $u\in L^1(\R^n)$ so
   that $u\in W^{1,q}(\R^n)\cap W^{2,q'}(\R^n)$ for all $q\in [1,\infty],q'\in (1,\infty)$ by
   Proposition~\ref{Prop Mapping properties}.
   
   \medskip
      
   Finally, for all $\phi\in C_c^\infty(\R^n)$ we get from $u\in W^{2,1}_{loc}(\R^n)$ and the
   definition of a weak derivative 
   $$
     \int_{\R^n} (-\Delta u + Vu)\phi\,dx 
     = \int_{\R^n} u(-\Delta\phi+V\phi)\,dx
     = \int_{\R^n} g(x,u)\phi \,dx,
   $$ 
   hence $-\Delta u + Vu = g(x,u)$ almost everywhere which proves that $u$ is a strong solution of
   \eqref{Gl II}.
   
   \qed

  \subsection{Proof of Theorem~\ref{Thm regularity},(1)}
   
  Our aim is to show that $u$ satisfies the assumptions of
  Proposition~\ref{Prop representation formula} so that we may infer the local regularity
  properties of $u$ from the representation formula $u=T_\omega(g_\omega)$ and the mapping
  properties of $T_\omega$. For our approach we first need to check that functions in
  $W_0^{2,\infty}(\R^n)$ with compact support are admissible test functions for
  \eqref{Gl II}.

\begin{proposition}\label{Prop Testing}
  Let the assumptions of Theorem \ref{Thm regularity},(1) hold. Then
  \begin{equation*}
    \int_{\R^n} u(-\Delta \phi + V(x) \phi)\,dx = \int_{\R^n} g(x,u)\phi \,dx
  \end{equation*}
   for all $\phi\in  W_0^{2,\infty}(\R^n)$ such that $\supp(\phi)$ is compact.
\end{proposition}
\begin{proof}
  Let $\phi\in W_0^{2,\infty}(\R^n)$ with compact support. By mollification we obtain a
  sequence $\phi_k\in C_0^\infty(\R^n)$ and a compact set $K$ such that
  $\supp(\phi),\supp(\phi_k)\subset K$, $\Delta \phi_k\to \Delta \phi$ pointwise almost
  everywhere in $K$, $|\Delta\phi_k|\leq \|\Delta\phi\|_\infty$ and $\phi_k\to\phi$ uniformly. The
  dominated convergence theorem gives
  \begin{align*}
    \int_{\R^n} g(x,u)\phi \,dx
    &= \int_K g(x,u)\phi \,dx \\
    &= \lim_{k\to \infty} \int_K g(x,u)\phi_k \,dx \\
    &= \lim_{k\to \infty} \int_K u (-\Delta \phi_k+ V(x)\phi_k) \,dx \\
    &= \int_K u (-\Delta \phi+ V(x) \phi) \,dx \\
    &= \int_{\R^n} u (-\Delta \phi+ V(x)\phi) \,dx.
  \end{align*}
\end{proof}
  
  In the following Proposition we verify the assumptions of Proposition \ref{Prop representation
  formula} in order to deduce $u=T_\omega(g_\omega)$.
  
  \begin{proposition}\label{Prop integrability positive solutions}
    Let the assumptions of Theorem \ref{Thm regularity},(1) hold. Then $u\in L^p(\R^n;\omega_0)$ for
    all $\omega_0>0$.
  \end{proposition}  
  \begin{proof}
    Let $J_{\frac{n-2}{2}}$ denote the Bessel function of the first kind of order $\frac{n-2}{2}$,
    let $v(r):= J_{\frac{n-2}{2}}(r)r^{\frac{2-n}{2}}$. Then $v$ lies
    in $C^\infty(\R)$ and is a classical radially symmetric solution of $-\Delta \psi=\psi$.
    Furthermore, there is $r_0>0$ such that $v$ is strictly decreasing in $[0,r_0)$ and $v'(r_0)=0,
    v(r_0)=:\kappa<0$. For $R>0$ the function $\phi_R$ defined by 
    $$
      \phi_R(x)=\phi(\frac{x}{R})\quad\mbox{ where }\quad \phi(x)= (v(x)-\kappa)\cdot 1_{\{|x|\leq r_0\}} 
    $$
    is positive in $B_{Rr_0}$ with $\supp(\phi_R)=\overline{B}_{Rr_0}$. By
    the choice of $\kappa$ we have $\phi\in C^{1,1}(\R^n)$ and Rademacher's theorem applied to $\partial_{x_i}\phi$,
    $i=1,\ldots,n$ shows that $\phi\in W_0^{2,\infty}(\R^n)$. Moreover, $\phi_R$ satisfies the
    differential equation
  \begin{equation*}
    -\Delta \phi_R + V(x)\phi_R =  \big((V(x)+\frac{1}{R^2})\phi_R -
    \frac{|\kappa|}{R^2}\big)\cdot 1_{\{|x|\leq Rr_0\}}
  \end{equation*}
  pointwise a.e. in $\R^n$. By Proposition \ref{Prop Testing} we may use $\phi_R$ as a test function in \eqref{Gl II}.
  Positivity of $u$ and $-\Delta \phi_R + V(x)\phi_R \leq (\|V\|_{L^\infty(\R^n)}+1) \phi_R$
  almost everywhere for all $R\geq 1$ implies 
  \begin{align} \label{Gl Nonneg 1}
    \int_{\R^n} (-\Delta\phi_R+V(x)\phi_R)u\,dx
    &\leq (\|V\|_{L^\infty(\R^n)}+1) \int_{B_{Rr_0}} \phi_Ru \,dx      \nonumber \\
    &\leq \frac{C_4}{2} \int_{B_{Rr_0}} u^p\phi_R \,dx + c \int_{B_{Rr_0}} \phi_R \,dx    \\   
    &\leq \frac{C_4}{2} \int_{B_{Rr_0}} u^p\phi_R \,dx + cR^n\,.  \nonumber 
  \end{align}
  where $C_4$ is the constant from \eqref{Gl growth cond I}. Here we used that for every $\eps>0$ there exists $C_\eps>0$
  such that $a\leq \eps a^p + C_\eps$ for all $a>0$. From the assumptions on $g$ we get
  \begin{align} \label{Gl Nonneg 2}
    \int_{\R^n} g(x,u)\phi_R \,dx 
    &\geq C_4 \int_{B_{Rr_0}} u^p \phi_R \,dx - C_3 \int_{B_Rr_0} \phi_R \,dx \nonumber \\
    &\geq C_4 \int_{B_{Rr_0}} u^p \phi_R \,dx - c R^n\,.
  \end{align}
  Subtracting \eqref{Gl Nonneg 2} from \eqref{Gl Nonneg 1} we get
  \begin{align}\label{Gl Nonneg 3}
     \frac{C_4}{2}\int_{B_{Rr_0}} u^p \phi_R \,dx     
     \leq c R^n \quad\mbox{ for all }R\geq 1.
  \end{align}   
  
  \smallskip
  
  For a fixed $\gamma\in (0,1)$ the function $\phi_R$ is uniformly
  bounded from below on $B_{\gamma Rr_0}$ so that \eqref{Gl Nonneg 3} implies
  \begin{align} \label{Gl Nonneg 4}
    \int_{B_{\gamma Rr_0}} u^p \,dx
    \leq c_\gamma\, \int_{B_{Rr_0}} u^p\phi_R\,dx
    \leq c_\gamma\, R^n \quad\mbox{ for all }R\geq 1. 
  \end{align} 
  Therefore we obtain the following inequality with $A_k:= \{k\gamma r_0\leq |x| < (k+1)\gamma r_0
  \}$, $k\in \N_0$ and $\omega_0>0$:
  \begin{align*}
    \int_{\R^n} e^{-\sqrt{\omega_0}|x|} u(x)^p\,dx    
    &= \sum_{k=0}^\infty \int_{A_k} e^{-\sqrt{\omega_0}|x|}u(x)^p \,dx \\
    &\leq c_\gamma\,\sum_{k=0}^\infty e^{-\gamma\sqrt{\omega_0} r_0 k } \int_{B_{\gamma r_0 (k+1)}}
    |u(x)|^p\,dx\\
	&\leq c_\gamma\,\Big(1+\sum_{k\geq \frac{1}{\gamma r_0}-1}^\infty
	e^{-\gamma\sqrt{\omega_0} r_0k } (\gamma r_0 (k+1))^n \Big) < \infty.
  \end{align*}     
  Hence, $u\in L^p(\R^n;\omega_0)$ for all $\omega_0>0$.
  \end{proof}
  
  
  \noindent
  {\it Proof of Theorem \ref{Thm regularity},(1):} Let $g$ satisfy \eqref{Gl growth cond I} and let $u\geq 0$ be a distributional solution. 
  Since $W:= V-\Gamma |u|^{p-1} \in L^\infty(\R^n)+L^{\frac{p}{p-1}}_{loc}(\R^n)$ and
  $\frac{p}{p-1}>\frac{n}{2}$ we find that $W$ lies in the local Kato class $K_n^{loc}$ (see
  \cite{sim},p.453) and thus Proposition~\ref{Prop decayToZero} 
  (applied to compact subsets of $\R^n$) gives $u\in L^\infty_{loc}(\R^n)$. From
  Proposition~\ref{Prop integrability positive solutions} we get $u\in L^p(\R^n;\omega)$ for all $\omega>0$ so that Proposition~\ref{Prop representation
  formula} implies $u=T_\omega(g_\omega)$ where $g_\omega\in L^\infty_{loc}(\R^n)\cap
  L^1(\R^n;\omega)$ for all $\omega>0$. Since $(k,q,r)=(1,\infty,\infty)$ is admissible for
  Proposition~\ref{Prop local Mapping properties} we obtain $u\in W_{loc}^{1,\infty}(\R^n)$, in particular 
   $$
     \int_{\R^n} \nabla u\nabla\phi + \omega u\phi\,dx 
     = \int_{\R^n} u(-\Delta\phi+\omega \phi)\,dx
     = \int_{\R^n} g_\omega \phi \,dx,
   $$ 
   for all $\phi\in C_0^\infty(\R^n)$ so that $u$ is a weak solution of the uniformly elliptic
   PDE \eqref{Gl IInew} . From $g_\omega\in L^\infty_{loc}(\R^n)$ we obtain $u\in
   W^{2,q}_{loc}(\R^n)$ for all $q\in [1,\infty)$ by Cald\'{e}ron-Zygmund estimates (cf. Chapter 9
   in Gilbarg, Trudinger \cite{giltru}). The same reasoning as in part (2) shows that $u$ must be a
   strong solution in $\R^n$.
  \qed  
  


\section{Appendix A}

  In the proof of Proposition~\ref{Prop Existence approximate solution} we use the following
  auxiliary lemma.

  \begin{lemma} \label{Lem extension}
  Let $0<c_0<1$ and $\rho\geq 1$ be given. Then for all $p>1$ there exists a radially symmetric
  positive function $u_2\in C^\infty(\R^n\setminus B_\rho)$ such that
  \begin{align}\label{Gl Def w2} 
    \begin{aligned}
      -\Delta u_2 + u_2 &= u_2^p &&\text{in }\R^n\setminus B_\rho \\
      u_2(x) &= c_0 &&\text{for }|x|=\rho \\
      u_2(x) &\to 0 &&\text{exponentially as $|x|\to\infty$}. 
    \end{aligned}    
  \end{align}
  Moreover the following inclusion holds
  $$
    0<v(|x|) \leq u_2(x)\leq c_0\, e^{-\sqrt{1-c_0^{p-1}}(|x|-\rho)}
    \quad\mbox{for all } |x|\geq \rho 
  $$
  where $v(r)= \kappa r^{\frac{2-n}{2}}K_{\frac{n-2}{2}}(r)$. Here $K_{\frac{n-2}{2}}$ denotes
  the modified Bessel function of second kind and $\kappa>0$ is chosen such that $v(\rho)=c_0$.
\end{lemma}
\begin{proof}   
  We first use the method of sub- and supersolutions to find a solution $w_{2,R}$ of the
  following auxiliary elliptic ODE boundary value problem
  \begin{align}\label{Gl Approx ODE}
    -w_{2,R}'' - \frac{n-1}{r}w_{2,R}' + w_{2,R} &= w_{2,R}^p \qquad\text{in }(\rho,R),\\
    w_{2,R}(\rho) = c_0, \quad w_{2,R}(R) &= v(R) \notag
  \end{align}
  for any given $R>\rho$. As a supersolution  of \eqref{Gl Approx ODE} we may take the constant
  function $c_0$ since $c_0\geq c_0^p$ and $c_0=v(\rho)>v(R)$ using the fact that $v$ is strictly
  decreasing. Since $v$ is positive and satisfies the boundary conditions as well as
  \begin{displaymath}
    -v''(r)-\frac{n-1}{r}v'(r)+v(r)=0 \qquad\text{in }(\rho,R)
  \end{displaymath}
  we may choose $v$ as a subsolution.
  Hence the method of sub- and supersolutions (cf. \cite{walter}, \S 16) applies
  and produces a classical solution $w_{2,R}$ of \eqref{Gl Approx ODE} with the additional property
  \begin{align}\label{Gl enclosure w2R}
    0<v(r)\leq w_{2,R}(r)\leq c_0<1 \qquad\text{for } r>\rho.
  \end{align}
  The function $w_{2,R}$ cannot attain a local maximum at any $r^*\in (\rho,R)$ since
  in this case we would have $0\leq -w_{2,R}''(r^*) = w_{2,R}(r^*) (w_{2,R}(r^*)^{p-1}-1)$ contradicting
  \eqref{Gl enclosure w2R}. This implies that $w_{2,R}$ is decreasing since otherwise
  there would be $\rho\leq r_1<r_2<R$ such that $w_{2,R}(r_1)<w_{2,R}(r_2)$. Using that there is no
  interior local maximum this would lead to $w_{2,R}(r_1)<w_{2,R}(r_2)\leq w_{2,R}(R)=v(R)$ in
  contradiction to $w_{2,R}(r_1) \geq v(r_1)> v(R)$ by \eqref{Gl enclosure w2R} and
  strict monotonicity of $v$.

  \medskip

  Since $w_{2,R}$ is decreasing we have $w_{2,R}'\leq 0$ and from \eqref{Gl Approx ODE} and
  $w_{2,R}<1$ we get $w_{2,R}''> 0$,  hence
  \begin{equation}\label{Gl enclosure w2R prime}
    0\geq w_{2,R}'(r)\geq w_{2,R}'(\rho)\geq v'(\rho) \quad \mbox{ for all } r\in [\rho,R].
  \end{equation}
  From \eqref{Gl Approx ODE}, \eqref{Gl enclosure w2R} and \eqref{Gl enclosure w2R prime} it follows
  that for all $R_0>\rho$ the families $(w_{2,R}')_{R>R_0}$, $(w_{2,R}'')_{R>R_0}$ are uniformly bounded with respect to $R$.
  By the Arzel\`{a}-Ascoli theorem, there is a sequence $(w_{2,R_j})$ with
  $\lim_{j\to\infty}R_j=\infty$ which converges uniformly along with its first derivatives on every compact subset of $[\rho,\infty)$ to some $\tilde u_2\in C^1([\rho,\infty))$ which satisfies the enclosure $0<v\leq \tilde u_2 \leq c_0<1$. Writing
  \begin{equation*}
    w_{2,R}(r)
    = c_0 + \frac{\rho}{2-n}\big( (\frac{\rho}{r})^{n-2} - 1\big) w_{2,R}'(\rho)
      + \int_\rho^r \int_\rho^s (\frac{t}{s})^{n-1} [w_{2,R}(t)-w_{2,R}(t)^{2p-1} ]\,dt\,ds
  \end{equation*}
  we obtain that $\tilde u_2=\lim_{R\to\infty} w_{2,R}$ belongs to $C^2([\rho,\infty))$ and solves the
  initial value problem
  \begin{align}\label{Gl Eq w2}
    -\tilde u_2'' - \frac{n-1}{r}\tilde u_2' + \tilde u_2 &= \tilde u_2^p \quad\text{in
    }(\rho,\infty),\qquad \tilde u_2(\rho) = c_0 
  \end{align}
  in the classical sense. In particular, $u_2(x):= \tilde u_2(|x|)$ defines a radially
  symmetric classical solution of problem \eqref{Gl Def w2} on $\R^n\setminus B_\rho$. It remains to show that
  $\tilde u_2$ decays exponentially at infinity.

  \medskip

  To this end we test \eqref{Gl Eq w2} with functions $\phi_k(r):= \phi(r-k)$ for $k>0$
  and $\phi\in C_0^\infty(\rho,\infty)$ arbitrary. Since $\tilde u_2\in C^2([\rho,\infty))$ is a
  decreasing function it has a limit $\tilde u_{2,\infty} := \lim_{r\to\infty} \tilde u_2(r)$ which
  satisfies $0\leq \tilde u_{2,\infty}<c_0<1$. Therefore the dominated convergence theorem implies
  \begin{align*}
    0
    &= \lim_{k\to\infty} \int_\rho^\infty \tilde u_2(r) \big(-\phi_k''(r)- \frac{n-1}{r}\phi_k'(r)
      + \phi_k(r)-\tilde u_2(r)^{p-1}\phi_k(r)\big)\,dr \\
    &= \lim_{k\to\infty} \int_\rho^\infty \tilde u_2(r+k) \big(-\phi''(r) -\frac{n-1}{r+k}\phi'(r) +
       \phi(r)-\tilde u_2(r+k)^{p-1}\phi(r)\big)\,dr \\
    &= \int_\rho^\infty \tilde u_{2,\infty} \big(-\phi''(r) + \phi(r)-\tilde u_{2,\infty}^{p-1}\phi(r)\big)\,dr \\
    &= \tilde u_{2,\infty}(1-\tilde u_{2,\infty}^{p-1}) \, \int_\rho^\infty \phi(r)\,dr
  \end{align*}
  and thus, $\phi$ being an arbitrary testfunction, we see that necessarily $\tilde u_{2,\infty}=0$.
    
  \medskip 
  
  Finally we to show $\tilde u_2\leq z$ where $z(r):= c_0 e^{-\sqrt{1-c_0^{p-1}}(r-\rho)}$.
  From $z''(r)=(1-c_0^{p-1})z$, $z(\rho)=c_0$  
  and $0<\tilde u_2\leq c_0, \tilde u_2'\leq 0$ we get
  \begin{align*}
    (\tilde u_2-z)''(r)
    &= -\frac{n-1}{r}\tilde u_2'(r) + \tilde u_2(r)(1-\tilde u_2(r)^{p-1})-(1-c_0^{p-1})z(r) \\
    &\geq (1-c_0^{p-1})(\tilde u_2-z)(r) \quad\mbox{ for all }r\geq \rho.
  \end{align*}
  which proves that $\tilde u_2-z$ cannot have any positive interior local maximum. Hence,
  \begin{displaymath}
    (\tilde u_2-z)(r)\leq \max\{0,(\tilde u_2-z)(\rho),(\tilde u_2-z)(\infty)\}=0
     \quad\mbox{ for all }r\geq \rho
  \end{displaymath}
  and the result follows.
\end{proof}

{\em Proof of Proposition~\ref{Prop Existence approximate solution}:}
  Let $n\geq 3$, choose $\rho$ such that the inequalities 
  $$
    \rho \geq 1, \qquad 
    \rho\geq \sqrt{\frac{4}{3}} \cdot \max \{ c_{n,q}^{\frac{q-1}{2}} : \frac{n}{n-2}\leq q\leq
    \frac{n+2}{n-2} \}
  $$ 
  hold true where $c_{n,p}$ is given by \eqref{Gl Def u1}. Then, given any $p\in
  (\frac{n}{n-2},\frac{n+2}{n-2})$ the choice $c_0 := c_{n,p}\rho^{-\frac{2}{p-1}}$ implies $0<c_0
  \leq c_{n,p}$ and $c_0^{p-1}\leq \frac{3}{4}$.
  
  \medskip
  
  Let now $u_2$ be given by Lemma \ref{Lem extension}, $u_1(x):=c_{n,p}|x|^{-\frac{2}{p-1}}$.
  Then the function $u_0$ defined in \eqref{Gl App sol} is positive radially
  symmetric and satisfies (i),(ii) by the choice of $u_1,u_2$.
  Moreover, $u_0\in C(\R^n\setminus\{0\})$ implies  $u_0\in H^1(\R^n\setminus B_\delta)$ for all
  $\delta>0$ and $u_1\in C^2(\overline{B}_{\rho}\setminus \{0\}),u_2\in C^2(\R^n\setminus B_\rho)$ gives (iii). 
  Property (iv) follows from the definition of $u_1$. 
  The explicit formula for $u_1$ and the enclosure of $u_2$ given by Lemma \ref{Lem extension} yield
  $$
    |\partial_\nu^+ u_0(x)| 
    = |\partial_\nu u_1(x)|
    \leq c\, c_{n,p}, \qquad
    |\partial_\nu^- u_0(x)| 
    = |\partial_\nu u_2(x)|
    \leq c\, c_{n,p}\quad (x\in \partial B_\rho)
  $$
  and we obtain (v). By the choice of $\rho$ we have $c_0^{p-1}\leq \frac{3}{4}$ so that Lemma
  \ref{Lem extension} gives the upper bound for $u_2(x) \leq c_0\, e^{-\frac{|x|-\rho}{2}}$ which
  shows (vi) and finishes the proof of Proposition~\ref{Prop Existence approximate solution}. \qed


\section{Appendix B}

The following proposition sums up two results from \cite{sim}.
  
  \begin{proposition}\label{Prop decayToZero}
    Let $\Omega=\R^n\setminus B_R$ for some $R\geq 0$. Let $W\in K_n(\Omega)$
    and assume $-\Delta u+Wu=0$ in $\Omega$ in the distributional sense where $u, Wu\in
    L^1_{loc}(\Omega)$. Then $u$ equals almost everywhere a continuous function in $\Omega$. 
    If in addition  $u\in L^q(\Omega)$ for some $q\in [1,\infty)$ then
    $u(x)\to 0$ as $|x|\to\infty$. 
  \end{proposition}
  \begin{proof}
    Continuity of $u$ follows from \cite{sim}, Theorem C.1.1. Moreover \cite{sim}, Theorem
    C.1.2. implies that for almost all $x\in\Omega$ with $\dist(x,\partial\Omega)>1$ we have 
    \begin{align*}
      |u(x)|
      \leq C(\|W_- \|_{K_n(B_1(x))}) \int_{B_1(x)} |u(y)|\,dy  
      \leq C(\|W_- \|_{K_n(\Omega)}) \int_{B_1(x)} |u(y)|\,dy.   
    \end{align*}    
    Now if $u\in L^p(\Omega)$ we have $\lim_{|x|\to\infty} \int_{B_1(x)} |u(y)|^p\,dy=0$
    and thus H\"{o}lder's inequality implies $\lim_{|x|\to\infty} \int_{B_1(x)} |u(y)|\,dy=0$.
    Hence the result.
  \end{proof} 
    
  \begin{proposition}\label{Prop exponentialdecay}
    Let $\Omega=\R^n\setminus B_R$ for some $R\geq 0$. Let $W\in L^s(\Omega)+L^\infty(\Omega)$
    for some $s>\frac{n}{2}$, and assume $0<\Sigma:= \inf \sigma_{ess}(-\Delta+W(x))$. If $u\in
    H^1_{loc}(\Omega)\cap L^q(\Omega)$ for some $q\in [2,\frac{2n}{(n-2)_+})$ is a weak solution of $-\Delta u + Wu=0$ in $\Omega$ then for all $\mu \in (0,\sqrt{\Sigma})$ there is a constant $C_\mu>0$ such that 
    $$
      |u(x)| \leq C_\mu e^{-\mu|x|}\qquad \text{for all } x\in\Omega \text{ with
      }\dist(x,\partial\Omega)>1. 
    $$
  \end{proposition}
  \begin{proof}    
    {\it 1st step: Proof of exponential integrability} 
            
    Let $\mu \in (0,\sqrt{\Sigma})$ be arbitrary and let $\chi\in C^\infty(\R^n)$ such that
    $\chi|_{B_1}\equiv 0$ and $\chi|_{B_2^c} \equiv 1$.  
    Let $\chi_s(x)= \chi(s^{-1}x)$ for $x\in\R^n$ and $s>0$. For $\rho>r>R$ we define the function     
    $$
      \chi_{r,\rho}:= \chi_r \cdot (1-\chi_\rho).
    $$    
    Notice that the support of $\chi_{r,\rho}$ is contained in the annulus
    $\overline{B}_{2\rho}\setminus B_r$ and $\chi_{r,\rho}\equiv  \chi_r$ on
    $\overline{B}_\rho$.    
    For $\sigma>0$ we define 
    $$
      \phi = \xi^2 u \quad\mbox{ where }\quad \xi(x) =
       \chi_{r,\rho}(x)e^{\frac{\mu|x|}{1+\sigma|x|}}. 
    $$     
    Since $u\in H^1_{loc}(\Omega)$ is a weak solution of \eqref{Gl I} in $\Omega$
    and $\supp(\chi_{r,\rho})\subset \overline{B}_{2\rho}\setminus B_r$ we have $\phi\in
    H_0^1(\Omega)$ and
    \begin{align*}
      0 
      &= \int_{\R^n} \nabla u\nabla \phi + W u\phi \,dx \\
      &= \int_{\R^n} |\nabla (\xi u)|^2 + W |\xi u|^2 - 
          |\nabla \xi|^2|u|^2 \,dx.    
  \end{align*}
  Now fix a $\delta\in (0,\frac{1}{2}(\Sigma-\mu^2))$.
  From $|\nabla \xi|\leq e^{\frac{\mu|x|}{1+\sigma|x|}}(|\nabla\chi_{r,\rho}|+\mu |\chi_{r,\rho}|)$ we infer  
  $$
    |\nabla \xi|^2\leq  (\mu^2+\delta) |\chi_{r,\rho}|^2 e^{\frac{2\mu|x|}{1+\sigma|x|}} + 
    (1+\mu^2\delta^{-1}) |\nabla\chi_{r,\rho}|^2 e^{\frac{2\mu|x|}{1+\sigma|x|}}.
  $$
  Hence,
  \begin{align} \label{Gl Exp decay I}
    0     
    &\geq \int_{\R^n} |\nabla (\xi u)|^2 + W |\xi u|^2 \,dx \\    
    &- (\mu^2+\delta) \int_{\R^n} |\chi_{r,\rho}|^2|u|^2e^{\frac{2\mu|x|}{1+\sigma|x|}} \,dx 
     -  (1+\mu^2\delta^{-1})\int_{\R^n} 
      |\nabla \chi_{r,\rho}|^2|u|^2e^{\frac{2\mu|x|}{1+\sigma|x|}} \,dx. \nonumber
  \end{align} 
  In view of $\inf \sigma_{ess}(W)\geq \Sigma$ and Persson's Theorem (cf. \cite{hislopsigal},
  Theorem 14.11.) we may choose $r>0$ so large that for all $\rho>r,\sigma>0$ the
  following inequality holds
  \begin{align} \label{Gl Exp decay II}
     \int_{\R^n} |\nabla (\xi u)|^2 +  W|\xi u|^2
     &\geq (\Sigma-\delta) \int_{\R^n} |\xi u|^2\,dx \\
     &= (\Sigma-\delta) \int_{\R^n} |\chi_{r,\rho}|^2 |u|^2e^{\frac{2\mu|x|}{1+\sigma|x|}}\,dx.
     \nonumber
  \end{align}
  From \eqref{Gl Exp decay I} and \eqref{Gl Exp decay II} we get for all $\rho>r,\sigma>0$
  \begin{equation}
    \int_{\R^n} \chi_{r,\rho}^2 |u|^2 e^{\frac{2\mu|x|}{1+\sigma|x|}}\,dx
    \leq \frac{1+\mu^2\delta^{-1}}{\Sigma-\mu^2-2\delta} \int_{\R^n} |\nabla \chi_{r,\rho}|^2|u|^2
    e^{\frac{2\mu|x|}{1+\sigma|x|}}\,dx.
  \label{kritisch}
  \end{equation}
  We want to take the limit $\rho\to \infty$. In the integral on the left-hand side of \eqref{kritisch} this can
  be done by the monotone convergence theorem. If $q=2$ then the right-hand side of \eqref{kritisch}
  can be treated by the dominated convergence theorem. In the case $2<q<\frac{2n}{(n-2)_+}$ notice that 
  \begin{align*}
    \int_{\R^n} (|\nabla \chi_{r,\rho}|^2-|\nabla\chi_r|^2)^{\frac{q}{q-2}}\,dx
    = \int_{\{\rho\leq |x|\leq 2\rho\}} |\nabla \chi_{\rho}|^{{\frac{2q}{q-2}}}\,dx 
    \leq \|\nabla \chi\|_\infty \rho^{n-{\frac{2q}{q-2}}}\to 0\quad\mbox{ as }\rho\to\infty.
  \end{align*}
  Hence \eqref{kritisch} holds with $\chi_{r,\rho}$ replaced by $\chi_r$.   
   Taking the limit $\sigma\to 0$ we obtain 
  $$ 
    \int_{\R^n} \chi_r^2 |u|^2 e^{2\mu|x|}\,dx \leq
    \frac{1+\mu^2\delta^{-1}}{\Sigma-\mu^2-2\delta} \int_{\R^n} |\nabla \chi_r|^2 |u|^2 e^{2\mu|x|}\,dx < \infty. 
  $$
  The right-hand side is finite since $\nabla \chi_r$ has compact support.
  Hence, $\chi_r u e^{\mu|x|}\in L^2(\R^n)$ and thus $u e^{\mu|x|}\in L^2(\R^n\setminus B_{2r})$.
  
  \medskip
  
  {\it 2nd step: Pointwise exponential decay} 
          
    Since $u$ is a weak solution of $-\Delta u + W(x)u
    = 0$ in $\Omega$ the Harnack inequality (cf. \cite{hanlin}, Theorem 4.1) implies that there is
    positive constant $C=C(\|W\|_{L^s(B_2(z))})$ such that 
    \begin{align}\label{Gl Est pointwiseexpdec}      
      \|u\|_{L^\infty(B_1(z))}
      &\leq C(\|W\|_{L^s(B_2(z))}) \|u\|_{L^2(B_2(z))}. 
    \end{align}    
    for all $z\in\R^n$ with $|z|>2r+2$.
    Since $W\in L^s(\Omega)+L^\infty(\Omega)$ the constant $C$ in \eqref{Gl Est pointwiseexpdec} is
    w.l.o.g. independent of $z$. Hence, we get
  \begin{align*}
    \|u e^{\mu|\cdot|}\|_{L^\infty(B_1(z))}
    &\leq \|u\|_{L^\infty(B_1(z))} \|e^{\mu|\cdot|}\|_{L^\infty(B_1(z))} \\
    &\leq C  \|u\|_{L^2(B_2(z))} e^{\mu(|z|+1)}\\
    &\leq C  \|u e^{\mu|\cdot|}\|_{L^2(B_2(z))} e^{-\mu(|z|-2)} e^{\mu(|z|+1)}\\ 
    &\leq Ce^{3\mu} \|u e^{\mu|\cdot|}\|_{L^2(\R^n\setminus B_r)}=: C_\mu.
  \end{align*}
  for $|z|>2r+2$ and thus $|u(x)|\leq C_\mu e^{-\mu|x|}$ for 
  $|x|>2r+1$. Moreover by Proposition~\ref{Prop decayToZero} $u$ is bounded outside a neigbhourhood of
  $\partial\Omega$ and thus
  $$
      |u(x)| \leq C_\mu e^{-\mu|x|}\qquad \text{for all } x\in\Omega \text{ with
      }\dist(x,\partial\Omega)>1. 
  $$
\end{proof}
    
\section{Acknowledgements}

The authors would like to thank Kazunaga Tanaka (Waseda Univ, Japan)
for interesting discussions leading to Remark 4,(2) and Dirk Hundertmark 
(KIT, Germany) for suggesting Agmon's method in the proof of Proposition \ref{Prop
exponentialdecay}.



\end{document}